\RequirePackage[l2tabu, orthodox]{nag}

\documentclass[
11pt,                          % standard font size
english                        % standard language
]{article}

\usepackage[english]{babel}    % with explicit language
\usepackage{amsmath}           % ams mathematical stuff
\usepackage[utf8]{inputenc}    % smart input of funny chars
\usepackage[T1]{fontenc}       % also for the font encoding
\usepackage{longtable}         % tables longer than one page
\usepackage{exscale}           % large summation signs in 11pt
\usepackage[final]{graphicx}   % to include pdf pictures
\usepackage[sort]{cite}        % nicer citations
\usepackage{eucal}
\usepackage{array}             % nice tables
\usepackage{wasysym}           % smiley symbols
\usepackage[a4paper]{geometry} % geometry of page layout
\usepackage{xspace}            % better spacing after macros
\usepackage{tikz}              % for commutative diagrams and stuff
\usepackage{tikz-cd}
\usepackage{amssymb}           % ams symbols
\usepackage{bbm}               % nicer bbm fonts
\usepackage{mathtools}         % some nice tricks for stackrel
\usepackage[amsmath,thmmarks,hyperref]{ntheorem} % nicer theorems
\usepackage{mathrsfs}          % nice calligraphic font
\usepackage{stmaryrd}          % more brackets
\usepackage{paralist}          % better enumerate lists

\usepackage[expansion=false    % no font expansion
           ]{microtype}        % only protrusion
\usepackage[nottoc]{tocbibind} % refs and index in the toc
\usepackage[%backref=page,      % backrefs in the bibliography
           final=true,         % always treat as final
           pdfpagelabels       % use pdf page labels
           ]{hyperref}         % hyperrefs are cool!

\usetikzlibrary{matrix}
\author{
  \textbf{Chiara Esposito}\thanks{\texttt{chiara.esposito@mathematik.uni-wuerzburg.de}},\\
    Institut für Mathematik \\
  Lehrstuhl für Mathematik X \\
  Universität Würzburg \\
  Campus Hubland Nord \\
  Emil-Fischer-Straße 31 \\
  97074 Würzburg \\
  Germany \\[0.3cm]
  \textbf{Niek de Kleijn}\thanks{\texttt{niekdekleijn@gmail.com }},\\
  Département de mathématiques\\
  Université libre de Bruxelles CP 218\\
  Boulevard du Triomphe\\ 
  1050 Bruxelles\\
  Belgium
    \\[0.3cm]
}

\renewcommand{\mathbb}[1]{\mathbbm{#1}}

\newcommand{\refitem}[1] {\textit{\ref{#1}.)}}

\numberwithin{equation}{section}

\allowdisplaybreaks

\renewcommand{\arraystretch}{1.2}

\let\originalleft\left
\let\originalright\right
\renewcommand{\left}{\mathopen{}\mathclose\bgroup\originalleft}
\renewcommand{\right}{\aftergroup\egroup\originalright}

\theoremheaderfont{\normalfont\bfseries}
\theorembodyfont{\itshape}
\newtheorem{lemma}{Lemma}[section]
\newtheorem{proposition}[lemma]{Proposition}
\newtheorem{theorem}[lemma]{Theorem}
\newtheorem{corollary}[lemma]{Corollary}
\newtheorem{definition}[lemma]{Definition}

\theorembodyfont{\rmfamily}

\newtheorem{example}[lemma]{Example}
\newtheorem{remark}[lemma]{Remark}

\makeatletter
\def\theorem@checkbold{}
\makeatother

\theoremheaderfont{\scshape}
\theorembodyfont{\normalfont}
\theoremstyle{nonumberplain}
\theoremseparator{:}
\theoremsymbol{\hbox{$\heartsuit$}}
\newtheorem{proof}{Proof}

\newenvironment{lemmalist}{\begin{compactenum}[\itshape i.)]}{\end{compactenum}}

\newenvironment{propositionlist}{\begin{compactenum}[\itshape i.)]}{\end{compactenum}}
\newenvironment{definitionlist}{\begin{compactenum}[\itshape i.)]}{\end{compactenum}}

\newcommand{\ring}[1]         {\mathsf{#1}}
\newcommand{\group}[1]        {\mathrm{#1}}
\newcommand{\algebra}[1]      {\mathscr{#1}}

\newcommand{\lie}[1]          {\mathfrak{#1}}

\newcommand{\Fun}[1][k]      {\mathscr{C}^{#1}}
\newcommand{\Cinfty}         {\Fun[\infty]}

\newcommand{\acts}            {\mathbin{\triangleright}}
\newcommand{\Sec}[1][k]      {\Gamma^{#1}}
\newcommand{\Secinfty}       {\Sec[\infty]}

\newcommand{\tensor}[1][{}]           {\mathbin{\otimes_{\scriptscriptstyle{#1}}}}
\newcommand{\Anti}                    {\Lambda}
\newcommand{\Sym}                     {\mathrm{S}}
\newcommand{\argument}       {\,\cdot\,}
\DeclareMathOperator{\id}    {\mathsf{id}}
\newcommand{\pr}             {\mathrm{pr}}
\newcommand{\Lie}   {\mathscr{L}}
\DeclarePairedDelimiter{\Schouten}{\llbracket}{\rrbracket}
\newcommand{\D}              {\mathop{}\!\mathrm{d}}
\newcommand{\End}            {\operatorname{\mathsf{End}}}

\newcommand{\hhbar}{[\![\hbar]\!]}
\newcommand{\Rf}  {\mathbb{R}^d_{\mathrm{formal}}}
\def\presuper#1#2% 
{\mathop{}%
   \mathopen{\vphantom{#2}}^{#1}%
   \kern-\scriptspace%
   #2}
\newcommand{\Dpoly}[2]{\presuper{#1}{\mathrm{D}_{\mathrm{poly}}^{#2}}}
\newcommand{\Tpoly}[2]{\presuper{#1}{\mathrm{T}_{\mathrm{poly}}^{#2}}}
\newcommand{\Cohom}[1]{\mathrm{H}^{#1}}
\newcommand{\FibT}[1]{\mathcal{T}_{\mathrm{poly}}^{#1}}
\newcommand{\FibD}[1]{\mathcal{D}_{\mathrm{poly}}^{#1}}
\newcommand{\FormsT}[3]{\presuper{E}{\Omega^{#1}(#2;\FibT{#3})}}
\newcommand{\FormsD}[3]{\presuper{E}{\Omega^{#1}(#2;\FibD{#3})}}
\newcommand{\conn}{\nabla\!\!\!\!\nabla}
\DeclareMathOperator{\Ker}{Ker}
\newcommand{\R}{\mathbb{R}}

\newcommand{\Z}{\mathbb{Z}}

\newcommand{\Q}{\mathbb{Q}}

\newcommand{\UE}{\algebra U}

\title{Universal Deformation Formula, Formality and Actions}

\author{
  \textbf{Chiara Esposito}\thanks{\texttt{chiara.esposito@mathematik.uni-wuerzburg.de}},\\
    Institut für Mathematik \\
  Lehrstuhl für Mathematik X \\
  Universität Würzburg \\
  Campus Hubland Nord \\
  Emil-Fischer-Straße 31 \\
  97074 Würzburg \\
  Germany \\[0.3cm]
  \textbf{Niek de Kleijn}\thanks{\texttt{niekdekleijn@gmail.com }},\\
  Département de mathématiques\\
  Université libre de Bruxelles CP 218\\
  Boulevard du Triomphe\\ 
  1050 Bruxelles\\
  Belgium
    \\[0.3cm]
}

\begin{document}

\maketitle

\begin{abstract}
  In this paper we provide a quantization via formality of
  Poisson actions of a triangular Lie algebra $(\lie g,
  r)$ on a smooth manifold $M$. Using the formality of 
  polydifferential operators on Lie
  algebroids we obtain a deformation quantization of $M$ together with
  a quantum group $\mathscr{U}_\hbar(\mathfrak{g})$ and a map of
  associated DGLA's. This motivates a definition of quantum action in
  terms of $L_\infty$-morphisms which generalizes the one given by
  Drinfeld.
\end{abstract}

\newpage

\tableofcontents

%
%Intro
%
\section{Introduction}

The concept of deformation quantization has been introduced by Bayen,
Flato, Fronsdal, Lichnerowicz and Sternheimer in their seminal paper
\cite{bayen.et.al:1978a} based on the theory of associative
deformations of algebras \cite{gerstenhaber:1964a}.  A formal star
product on a Poisson manifold $M$ is defined as a formal associative
deformation of the algebra of smooth functions $\Cinfty (M)$
on $M$ (the name comes from the notation $\star$ for the deformed product) and its existence has been proved as a corollary of the
so-called \emph{formality theorem} in \cite{kontsevich:2003a} (for
more details in deformation quantization we refer to the textbooks
\cite{esposito:2015a, waldmann:2007a}). On the other hand, Drinfeld
introduced the notion of quantum groups in the setting of formal
deformations, see e.g. the textbooks \cite{chari.pressley:1994a,
  etingof.schiffmann:1998a} for a detailed discussion.  Drinfeld also
introduced the idea of using symmetries to get formal
deformations. More explicitely, given an action by derivations of a
Lie algebra $\lie{g}$ on an associative algebra $(\algebra{A},
m_{\algebra A})$, the definition of the so-called \emph{Drinfeld
  twist} \cite{drinfeld:1983a, drinfeld:1988a} $J \in
(\algebra{U}({\lie{g}}) \tensor \algebra{U}(\lie g))\hhbar$ allows us
to obtain an associative formal deformation of $\algebra{A}$ by means
of a \emph{universal deformation formula}
\begin{equation}
    \label{eq:TheUDF}
    a \star_{J} b
    =
    m_{\algebra A} (J\acts (a \tensor b))
\end{equation}
for $a, b \in \algebra{A}\hhbar$. Here $\acts$ is the action of
$\lie{g}$ extended to the universal enveloping algebra
$\algebra{U}(\lie g)$ and then to $\algebra{U}(\lie g) \tensor
\algebra{U}(\lie g)$ acting on $\algebra{A} \tensor \algebra{A}$.  The
deformed algebra $(\algebra A\hhbar, \star_J)$ is then a
module-algebra for the quantum group:
\begin{equation}
   \UE_J(\lie{g}) 
   := (\UE(\lie{g})\hhbar, \Delta_J 
   := J \Delta J^{-1}).\end{equation}
In other words, Drinfeld obtains a quantized action. 
We mention here that the relevance of deformations induced via symmetries has been deeply investigated in
\cite{giaquinto.zhang:1998a} and in a non-formal setting in
\cite{bieliavsky.gayral:2015a}. 
%\textcolor{red}{something about formality and actions, literature}

The aim of this paper consists in obtaining a more general notion of
deformation through symmetry, by using formality theory. We focus on
the quantization of Lie algebra actions in the particular case of
triangular Lie algebras. Such actions can be regarded as the
infinitesimal version of Poisson Lie group actions (see e.g.
\cite{Kosmann-Schwarzbach2004,Semenov1985}) and they are very
important in the context of integrable systems. Triangular Lie
algebras and their quantizations have been studied by many authors,
see e.g. \cite{Calaque2006,enriquez.etingof:2005,xu:2002a}.
The idea of applying formality to actions has also been used 
in \cite{Arnal2007}, where
the authors use the Kontsevich formality on a Poisson manifold to construct
for each Poisson vector field a derivation of the star product. We recover this result.
%\cite{Sharygin2016} seems the same...

The formality theorem states the existence of an
$L_\infty$-quasi-isomorphism from polyvectorfields to polydifferential
operators on a manifold $M$. In \cite{Dolgushev2005, Dolgushev2005a}
Dolgushev proves the theorem for general $M$ using the proof for
$M=\R^n$. In order to construct such $L_\infty$-quasi-isomorphisms,
Dolgushev uses Fedosov's methods \cite{fedosov:1994a} concerning
formal geometry, Kontsevich's quasi-isomorphism
\cite{kontsevich:2003a} and the twisting procedure inspired by Quillen
\cite{Quillen69}. Following the construction provided by Dolgushev,
Calaque proved a formality theorem for Lie algebroids
\cite{Calaque2005}.
We consider an infinitesimal action of $\lie g$ on $M$, i.e. a Lie
algebra homomorphism $\varphi\colon
\mathfrak{g}\rightarrow\Secinfty(TM)$. This can immediately be
extended to a DGLA morphism
\begin{equation}
  \Tpoly{\mathfrak{g}}{} \longrightarrow \Tpoly{}{}(M),
\end{equation}
where $\Tpoly{\mathfrak{g}}{} = \wedge^\bullet \lie g$ and
$\Tpoly{}{}(M) = \Secinfty(\wedge^\bullet TM)$  with the brackets extended via a Leibniz rule. From the formality
theorem we know that we have the following $L_\infty$-quasi-isomorphisms
\begin{equation}
  \Tpoly{\mathfrak{g}}{} \longrightarrow \Dpoly{\mathfrak{g}}{}
  \quad \mbox{and} \quad
  \Tpoly{}{}(M) \longrightarrow \Dpoly{}{}(M).
\end{equation}
Using the quasi-invertibility of $L_\infty$-quasi-isomorphisms we obtain
the existence of an $L_\infty$-morphism 
\begin{equation}
\label{eq:totwist}
  \Dpoly{\mathfrak{g}}{} \longrightarrow \Dpoly{}{}(M).
\end{equation}
If the Lie algebra $\lie g$ is endowed with an $r$-matrix, i.e.  an
element $r \in \lie g \wedge \lie g$ satisfying the Maurer--Cartan
equation $\Schouten{r, r} = 0$, the action always induces a Poisson
structure on $M$ and it is automatically a Poisson action.
\begin{lemma}
  \leavevmode
  \begin{lemmalist}
  \item Given the formal Maurer--Cartan element $\hbar r \in
    \Tpoly{\mathfrak{g}}{}\hhbar$, we obtain via formality a
    Maurer--Cartan element $\rho_\hbar$. This yields a quantum group
    $\algebra U_{\rho_\hbar} (\lie g)$ with deformed coproduct $\Delta_{1\otimes 1+\rho_\hbar}$.
  \item Given the  Maurer--Cartan element $\hbar \pi =
    \varphi\wedge \varphi(r) \in \Tpoly{}{}(M)\hhbar$, we obtain via
    formality a Maurer--Cartan element $B_\hbar$. This induces a
    formal deformation $(\Cinfty(M)\hhbar, \star_{B_\hbar})$ of the Poisson
    algebra $(\Cinfty(M), \pi)$.
  \end{lemmalist}
\end{lemma}
The DGLA obtained from twisting the DGLA of $\lie g$-polydifferential
operators on the point $\Dpoly{\lie g}{}$ as given in \cite{Calaque2005} turns
out to be a special case of a DGLA canonically associated to any Hopf algebra
$H$, which we call $H_{poly}$. Given any Maurer-Cartan element $F\in H_{poly}$ 
there is the associated Drinfeld twist $J=1\otimes 1 + F$ (in the formal
sense). It turns out that the twisted DGLA $H_{poly}^F$ is canonically
isomorphic to $(H_J)_{poly}$ where $H_J$ denotes the Hopf algebra twisted by
$J$. Thus, using the twisting procedure on the $L_\infty$-morphism
\eqref{eq:totwist} we prove the following theorem.

\begin{theorem}
  Let $\lie g$ be a Lie algebra endowed with a classical $r$-matrix
  and a Lie algebra action $\varphi\colon
  \mathfrak{g}\rightarrow\Secinfty(TM)$ inducing a Poisson structure
  on $M$ by $\pi := \varphi\wedge\varphi(r)$.  Then, there exists an
  $L_\infty$-morphism $(\UE_F(\lie g)\hhbar)_{poly} \to
  C(\algebra A_\hbar;\algebra A_\hbar)$ between the DGLA associated to the quantum group $\UE_{\rho_\hbar}(\lie g)\hhbar$ and the Hochschild complex of the
  deformation quantization $\algebra A_\hbar$ of $\Cinfty(M)$.
\end{theorem}
This theorem motivates a definition, which generalizes 
Drinfeld quantized action. 
 \begin{definition}[Deformation Symmetry]
   A deformation symmetry of a Hopf algebra $H$ in 
   a unital associative algebra $\algebra A$ is a map
   \begin{equation} 
     \Phi
     \colon 
     H_{poly}\longrightarrow C(\algebra A)
   \end{equation}
   of $L_\infty$-algebras. 
\end{definition}
Comparing the quantized structures obtained with our approach, it is 
easy to see that we recover Drinfeld's universal deformation formulas. 

\vspace{0.3cm}

The paper is organized as follows. In Section~\ref{sec:Preliminaries}
we recall the language of $L_\infty$-algebras and the theorem, due to
Kontsevich, stating the existence of an $L_\infty$-quasi-isomorphism
between polyvector fields and polydifferential operators on the formal
completion at $0\in \R^d$.  In Section~\ref{sec:formality} we briefly
discuss the proof of formality for Lie algebroids, following 
\cite{Calaque2006, Dolgushev2005, Dolgushev2005a}. In particular, we 
recall the twisting procedure in the curved context.
Section~\ref{sec:FormalityActions}
contains the main
results of the paper, i.e. the construction of an $L_\infty$-morphism
out of a Poisson action and the discussion on twisted structures and
deformation symmetry. Finally we compare our approach with Drinfeld's
deformation formulas.

\section*{Acknoledgments}
The authors are grateful to Ryszard Nest and Stefan Waldmann 
for the inspiring discussions. 
%Thanks to Iakovos Androulidakis for the encouragment 

%%%%%%%%%%%%%%%%%%%%%%%%%%%%%%%%%%%%%%%%%%%%%%%%%%%%%%%%%%%%%%%%%%%%%%%%%%%%%%%%%%%%%%%%%%%%%%%%%%%%%%%%%%%%

%
%Preliminaries
%

\section{Preliminaries}
\label{sec:Preliminaries}

Given a graded vector space $V^\bullet$ over $\mathbb{K}$ 
we denote the $k$-\emph{shifted} vector space by $V[k]^\bullet$, it
is given by
\begin{equation}
  V^\bullet[k]^l=V^{l+k}
\end{equation}

%%%%%%%%%%%%%%%%%%%%%%%%%%%%%%%%%%%%%%%%%%%%%%%%%%%%%%%%%%%%%%%%%%%%%%%%%%%%%%%%%%%%%%%%%%%%%%%%%%%%%%%%%%%%%

\subsection{$L_\infty$-setting}
\label{sec:Linfty}

We shall recall the definitions of $L_\infty$-algebra and
$L_\infty$-morphisms for the convenience of the reader (and to fix
certain conventions). For the rest of this section we consider a
field $\mathbb{K}$ of characteristic $0$. Although many constructions 
will also allow for replacement of $\mathbb{K}$ by a PID containing the rationals.
\begin{definition}[$L_\infty$-algebra]
  A degree $+1$ coderivation $Q$ on the co-unital conilpotent
  cocommutative coalgebra $S^c(\mathfrak{L})$ cofreely cogenerated by
  the graded vector space $\mathfrak{L}[1]^\bullet$ over $\mathbb{K}$ is
  called an $L_\infty$-structure on the graded vector space $\mathfrak{L}$ if
  $Q^2=0$.
\end{definition}
In more explicit terms we have
\begin{equation}
  S^c(\mathfrak{L})=\bigoplus_{k=0}^\infty \Anti^{k}
  \left(\mathfrak{L}[1]\right)
\end{equation}
equipped with the coproduct $\Delta$ given by 
\begin{align}
  \Delta(1)
  &=
  1\otimes 1
  \qquad \mbox{and}
  \\
  \Delta(\gamma_1\wedge\ldots\wedge \gamma_k)
  &=
  1\otimes\gamma_1\wedge\ldots\wedge\gamma_k
  +
  \gamma_1\wedge\ldots\wedge\gamma_k\otimes 1 
  +
  \overline{\Delta}(\gamma_1\wedge\ldots\wedge\gamma_k)
\end{align}
for $k\geq1$ and any $\gamma_i\in \mathfrak{L}[1]$. Here we have
\begin{equation}
  \overline{\Delta}(\gamma_1\wedge\ldots\wedge \gamma_k)=
  \sum_{i=1}^{k-1}\sum_{\sigma\in\mbox{\tiny Sh($i$,$k-i$)}}\epsilon(\sigma)
  \gamma_{\sigma(1)}\wedge\ldots\wedge \gamma_{\sigma(i)}\bigotimes
  \gamma_{\sigma(i+1)}\wedge\ldots\wedge \gamma_{\sigma(k)},
\end{equation}
where $\mathrm{Sh}(i,k-i)$ denotes the $(i,k-i)$ shuffles in the symmetric
group $S_k$ in $k$ letters and the \emph{Koszul sign}
$\epsilon(\sigma)=\epsilon(\sigma, \gamma_1,\ldots, \gamma_k)$ is
determined by the rule
\begin{equation}
  \gamma_1\wedge\ldots\wedge\gamma_k=\epsilon(\sigma)
  \gamma_{\sigma(1)}\wedge\ldots\wedge\gamma_{\sigma(k)}.
\end{equation}
Recall that $S^c(\mathfrak{L})$ is given by the co-invariants of the
tensor algebra for the action of the symmetric groups generated by
\begin{equation}
  (i \ i+1)(\gamma_1\otimes\ldots\otimes \gamma_k)
  =
  (-1)^{|\gamma_i||\gamma_{i+1}|}\gamma_{1}\otimes\ldots\otimes 
  \gamma_{i-1}\otimes \gamma_{i+1}\otimes \gamma_i\otimes \gamma_{i+2}\otimes\ldots \gamma_{k},
\end{equation}
where we use the vertical bars to denote the shifted degree, i.e. the
degree of $\gamma_i$ in $\mathfrak{L}[1]$. The co-unit is given by the
projection $\pr_\mathbb{K}$ onto the ground field $\mathbb{K}$.
\begin{remark}
  \label{rem:bialgebra}
  A direct computation shows that, denoting the flip $a\otimes
  b\mapsto (-1)^{|a||b|}b\otimes a$ by $\tau$, we have
  \begin{equation}
    \Delta\circ
    (\argument\wedge\argument)
    =
    (\argument\wedge\argument)
    \otimes
    (\argument\wedge\argument)  
    \circ
    (\id\otimes\tau\otimes\id)\circ\Delta\otimes\Delta.
  \end{equation}
  So we obtain the unital and co-unital bialgebra
  $(S^c(\mathfrak{L}),\argument\wedge\argument, 1\in \mathbb{K},
  \Delta,\pr_\mathbb{K})$, i.e. $1\wedge X=X$ for all $X\in
  S^c(\mathfrak{L})$.  We sometimes abuse notation by omitting
  $\wedge$ in favor of simple concatenation or superscripts, e.g.  $ab
  := a\wedge b$ and $x^3 := x\wedge x\wedge x$.
\end{remark}
\begin{lemma}[Characterization of coderivations]
  Every degree $+1$ coderivation $Q$ on $S^c(\mathfrak{L})$ is uniquely
  determined by the components
  \begin{equation}
    Q_n\colon \Anti^n(\mathfrak{L}[1])\longrightarrow \mathfrak{L}[2]
  \end{equation}
  by the formula 
  \begin{equation}
    Q(\gamma_1\wedge\ldots\wedge\gamma_n)
    =
    \sum_{k=0}^n\sum_{\sigma\in\mbox{\tiny Sh($k$,$n-k$)}}
    \epsilon(\sigma)Q_k(\gamma_{\sigma(1)}\wedge\ldots\wedge
    \gamma_{\sigma(k)})\wedge\gamma_{\sigma(k+1)}\wedge
    \ldots\wedge\gamma_{\sigma(n)},
  \end{equation} 
  where we use the conventions that $\mathrm{Sh}(n,0) = \mathrm{Sh}(0,n) = \{\id\}$
  and that the empty product equals the unit. 
\end{lemma}
\begin{proof}
  It follows by simply writing out both sides of the defining equation
  \begin{equation}
    \Delta\circ Q
    =
    (Q\otimes \id+\id\otimes Q)\circ\Delta.
  \end{equation}
\end{proof}
Note that $Q_0(1)$ is of degree $1$ in $\mathfrak{L}[1]$ (thus of
degree $2$ in $\mathfrak{L}$).  The condition $Q^2=0$ can now be
expressed in terms of a quadratic equation in the components $Q_n$.
\begin{example}[Curved Lie algebras]
  \label{ex:CurvedLie}
  Our main example of an $L_\infty$-algebra is given by (curved) Lie
  algebras, i.e. the tuple $(\mathfrak{L},R,\D,[\argument,\argument])$ where
  we set $Q_0(1) = R$, $Q_1 = \D$, $Q_2 = [\argument,\argument]$ and $Q_i=0$
  for all $i\geq 3$. The condition $Q^2=0$ amounts to:
  \begin{itemize} 
  \item $\D R=0$, 
  \item $\D^2(\argument)=[R,\argument]$,
  \item $\D$ is a derivation of $[\argument,\argument]$,
  \item The graded Jacobi identity for $[\argument,\argument]$.
  \end{itemize}
\end{example}
\begin{remark}
  We should note that the our definition of $L_\infty$-algebra is
  usually called \emph{curved} $L_\infty$-algebra (see e.g.~
  \cite{Markl100}).  Although this definition is also not set in
  stone, see for instance \cite{GradyGwilliam} for yet another notion
  of curved $L_\infty$-algebra.  For the purpose of this paper it is,
  however, more convenient to call the curved version simply
  $L_\infty$-algebra.  The only $L_\infty$-algebras playing a role in this paper are, however, the flat $L_\infty$-algebras, i.e.
  those having $Q_0=0$.  The usual definition for an
  $L_\infty$-algebra thus coincides with our definition of flat
  $L_\infty$-algebra.
\end{remark}
\begin{remark}
  \label{rem:filtration}
  In the following we have to deal with various infinite sums. In
  order for this to make sense, we always consider only
  $L_\infty$-algebras $\mathfrak{L}$ that are equipped with a
  decreasing filtration
  \begin{equation}
    \mathfrak{L}
    =
    \mathcal{F}^0\mathfrak{L}
    \supset\mathcal{F}^1\mathfrak{L}
    \supset\ldots\supset\mathcal{F}^k\mathfrak{L}\supset\ldots,
  \end{equation}
  respecting the $L_\infty$-structure and which is moreover
  \emph{complete}, i.e.
  \begin{equation}
    \bigcap_k\mathcal{F}^k\mathfrak{L}
    =
    \{0\}.
  \end{equation}
  This yields a corresponding complete metric topology and we consider
  convergence of infinite sums in terms of this topology.
\end{remark}
\begin{definition}[$L_\infty$-morphisms]
  \label{def:Linftymorph}
  Let $\mathfrak{L}$ and $\widetilde{\mathfrak{L}}$ be
  $L_\infty$-algebras.  A degree $0$ filtration preserving co-unital
  co-algebra morphism
  \begin{equation}
    F\colon 
    S^c(\mathfrak{L})
    \longrightarrow 
    S^c(\widetilde{\mathfrak{L}})
  \end{equation}
  such that $FQ = \widetilde{Q}F$ is called
  an $L_\infty$-morphism.
\end{definition}
\begin{lemma}[Characterization of co-algebra morphisms]
  A co-algebra morphism $F$ from\\ $S^c(\mathfrak{L})$ to
  $S^c(\widetilde{\mathfrak{L}})$ is uniquely determined by its
  components, also called Taylor coefficients,
  \begin{equation}
    F_n
    \colon 
    \Anti^n(\mathfrak{L}[1])
    \longrightarrow 
    \widetilde{\mathfrak{L}}[1],
  \end{equation}
  where $n\geq 1$. Namely, we set $F(1)=1$ and use the formula 
  \begin{equation}
    \label{eq:coalgebramorphism}
    \begin{gathered}
      F(\gamma_1\wedge\ldots\wedge\gamma_n)= 
      \\     
      \sum_{p\geq1}\sum_{\substack{k_1,\ldots, k_p\geq1\\k_1+\ldots+k_p=n}}
      \sum_{\sigma\in \mbox{\tiny Sh($k_1$,..., $k_p$)}}\frac{\epsilon(\sigma)}{p!}
      F_{k_1}(\gamma_{\sigma(1)}\wedge\ldots\gamma_{\sigma(k_1)})\wedge\ldots\wedge 
      F_{k_p}(\gamma_{\sigma(n-k_p+1)}\wedge\ldots\wedge\gamma_{\sigma(n)}),
    \end{gathered}
  \end{equation}
  where $Sh(k_1,...,k_p)$ denotes the set of $(k_1,\ldots,
  k_p)$-shuffles in $S_n$ and $Sh(n) = \{\id\}$.
\end{lemma}
\begin{proof} 
  It simply follows by writing out the defining equation
  \begin{equation}
    \Delta\circ F
    =
    F\otimes F\circ \Delta.
  \end{equation}
\end{proof}
\begin{example} 
  Let $(\mathfrak{L},R,\D,[\argument,\argument])$ and
  $(\mathfrak{L}',R',\D',[\argument,\argument]')$ be two curved Lie
  algebras and consider the morphism $C\colon
  \mathfrak{L}\longrightarrow \mathfrak{L}'$ of curved Lie algebras,
  i.e. $C(R)=R'$, $C\D = \D'C$ and $C$ is a morphism of the underlying
  Lie algebras.  Then the map $F$ given by applying the formula
  \eqref{eq:coalgebramorphism} to the components $F_1=C$ and $F_i=0$
  for $i\geq 2$ is an $L_\infty$-morphism. In general, if $F\colon
  \mathfrak{L}\longrightarrow \mathfrak{L}'$ is an
  $L_\infty$-morphism, then $F_1(R)=R'$, but we only have
  $\D'F_1(\gamma)=F_1\D(\gamma)+F_2(R\wedge\gamma)$.
\end{example}
Note that, given an $L_\infty$-morphism of \emph{flat}
$L_\infty$-algebras $\mathfrak{L}$ and $\widetilde{\mathfrak{L}}$, we
obtain the map of complexes
\begin{equation}
  F_1
  \colon 
  (\mathfrak{L},Q_1)
  \longrightarrow 
  (\widetilde{\mathfrak{L}},\widetilde{Q}_1).
\end{equation}
\begin{definition}[$L_\infty$-quasi-isomorphism]
  \label{def:Linftyquis}
  An $L_\infty$-morphism $F$ is called $L_\infty$-quasi-isomorphism if
  $F_1$ is a quasi-isomorphism of complexes.
\end{definition}
The $L_\infty$-quasi-isomorphisms we deal with in this paper happen to
be the ones witnessing \emph{formality}, let us therefore introduce
the notion of formal $L_\infty$-algebras here.
\begin{definition}[Formal $L_\infty$-algebra]
  \label{defformal}
  An $L_\infty$-algebra $\mathfrak{L}$ is called formal if it is flat
  and admits an $L_\infty$-quasi-isomorphism
  \begin{equation}
    F
    \colon 
    \mathrm{H}(\mathfrak{L})\longrightarrow \mathfrak{L}
  \end{equation}
  for the $L_\infty$-structure canonically induced on the cohomology
  $\mathrm{H}(\mathfrak{L})$ of $\mathfrak{L}$.
\end{definition}
Finally, a crucial concept for this paper is the one of Maurer--Cartan
elements, that we define below.
\begin{definition}[Maurer-Cartan element]
  \label{MCdef} 
  Given an $L_\infty$-algebra $(\mathfrak{L}, Q)$, an element $\pi\in
  \mathcal{F}^1\mathfrak{L}[1]^0$ is called a Maurer-Cartan or MC
  element if it satisfies the following equation
  \begin{equation}
    \label{eq:MC}
    \sum_{n=0}^\infty \frac{Q_n(\pi^n)}{n!}=0.
  \end{equation}
\end{definition}
%

%%%%%%%%%%%%%%%%%%%%%%%%%%%%%%%%%%%%%%%%%%%%%%%%%%%%%%%%%%%%%%%%%%%%%%%%%%%%%%%%%%%%%%%%%%%%%%%%%%%%%%%%%%%%%%%%%%%%%%%

\subsection{Local Formality}
\label{sec:LocalFormality}

Let us denote the formal completion at $0\in \R^d$ by $\Rf$. The
smooth functions $\Cinfty(\Rf)$ on $\Rf$ are given by the algebra
\begin{equation}
  \Cinfty(\Rf)
  :=
  \varprojlim_{k\rightarrow \infty} \Cinfty(\R^d)/\mathcal{I}_0^k,
\end{equation} 
where $\mathcal{I}_0$ denotes the ideal of functions vanishing at
$0\in\R^d$.  Note that $\Cinfty(\Rf)$ comes equipped with the complete
decreasing filtration
\begin{equation}
  \Cinfty(\Rf)\supset\mathcal{I}_0\supset\mathcal{I}_0^2\supset\ldots
\end{equation}
and its corresponding (metric) topology. The Lie algebra of continuous
derivations of $\Cinfty(\Rf)$ is denoted by
$\Tpoly{}{0}(\Rf)$. By setting $\Tpoly{}{-1} := \Cinfty(\Rf)$ we obtain
the Lie--Rinehart pair $(\Tpoly{}{-1},\Tpoly{}{0})$ and the graded
vector space
\begin{equation}
  \Tpoly{}{} (\Rf)
  :=
  \bigoplus_{k\geq -1}\Tpoly{}{k} (\Rf),
\end{equation}
where $\Tpoly{}{k} (\Rf) := \Anti^{k+1}\Tpoly{}{0} (\Rf)$ for $k\geq
0$.  Here the tensor product is understood to be over
$\Tpoly{}{-1}(\Rf)$ and completed.  Notice that there is no confusion
about grading here although it may seem unnatural at first glance. It
is actually obtained by shifting the natural grading.  The natural
structure is that of \emph{Gerstenhaber algebra}, but we are only
considering the underlying graded Lie algebra.  The Lie bracket
$\Schouten{\argument,\argument}$ on $\Tpoly{}{0}(\Rf)$ extends to a
graded Lie algebra structure on $\Tpoly{}{}(\Rf)$ by the rules
\begin{equation}
  \begin{aligned}
    \Schouten{f,g} 
    &= 
    0 ,
    \\
    \Schouten{X_0,f}  
    &=  
    X_0(f), 
    \\
    \Schouten{X_0\wedge\ldots\wedge X_k,Y} 
    &=      
    \sum_{j=0}^k(-1)^{kl+j}\Schouten{X_j,Y}\wedge X_0\wedge\ldots\wedge
    X_{j-1}\wedge X_{j+1}\wedge\ldots\wedge X_k
  \end{aligned}
\end{equation} 
for all $f,g \in \Tpoly{}{-1}(\Rf)$, $X_0,\ldots, X_k \in
\Tpoly{}{0}(\Rf)$ and $Y \in \Tpoly{}{l}(\Rf)$.

The universal enveloping algebra of the Lie-Rinehart pair
$(\Tpoly{}{-1}(\Rf),\Tpoly{}{0}(\Rf))$ is denoted by
$\Dpoly{}{0}(\Rf)$. Recall that $\Dpoly{}{0}(\Rf)$ is naturally
equipped with the structures of a bialgebra (see
e.g. \cite{Moerdijk2010}). More precisely, $\Dpoly{}{0}(\Rf)$ allows
an $\R$-algebra structure $\cdot$ and an $\R$-coalgebra structure
$\Delta$. We extend the algebra structure in the obvious
(componentwise) way to
\begin{equation}
  \Dpoly{}{}(\Rf)
  :=
  \bigoplus_{k\geq -1}\Dpoly{}{k}(\Rf),
\end{equation} 
where $\Dpoly{}{-1} (\Rf):= \Tpoly{}{-1}(\Rf)$ and $\Dpoly{}{k} (\Rf)
:= \left(\Dpoly{}{0}(\Rf)\right)^{\otimes k+1}$. Again the
tensor product is understood to be over $\Dpoly{}{-1}(\Rf)$ and
completed. This allows us to define two $\R$-bilinear operations
$\bullet$ and $[\argument,\argument]_G$ given by
\begin{equation}
\label{eq:Pre-Lie}
  P_1\bullet P_2
  :=
  \sum_{i=0}^{k_1}
  (-1)^{ik_2}(\id^{\otimes i}\otimes\Delta^{(k_2)}\otimes \id^{\otimes k_1-i})(P_1)\cdot (1^{\otimes i}\otimes P_2\otimes 1^{\otimes k_1-i})
\end{equation}
and 
\begin{equation}
 \label{eq:Lie}
  [P_1,P_2]_G
  :=
  P_1\bullet P_2-(-1)^{k_1k_2}P_2\bullet P_1
\end{equation}
where $P_1\in \Dpoly{}{k_1}(\Rf)$, $P_2\in\Dpoly{}{k_2}(\Rf)$ and
$\Delta^{(k)}$ denotes the $k$-th iteration of $\Delta$ given by 
$(\Delta\otimes\id^{\otimes k-1})(\Delta\otimes \id^{\otimes
  k-2})\ldots(\Delta\otimes \id)\Delta$. Note that the
bracket $[\argument,\argument]_G$ defines a graded Lie algebra
structure on $\Dpoly{}{}(\Rf)$.
\begin{theorem}[Kontsevich\cite{kontsevich:2003a}]
  \label{thm:kontsevich}
  There exists an $L_\infty$-quasi-isomorphism between DGLA's
  \begin{equation}
    \label{eq:kontsevich}
    \mathscr{K}
    \colon 
    \left(\Tpoly{}{}(\Rf),0,\Schouten{\argument,\argument}\right)
    \longrightarrow 
    \left(\Dpoly{}{}(\Rf),\partial,[\argument,\argument]_G\right)
  \end{equation} 
  where $\partial = [\mu,\argument]_G$ for $\mu=1\otimes 1\in
  \Dpoly{}{1}(\Rf)$. Moreover
  \begin{enumerate} 
  \item $\mathscr{K}$ is $\group{GL}(d,\R)$ equivariant;
  \item\label{thm:kontsevich2} $\mathscr{K}_n(X_1,\ldots, X_n)=0$ for all $X_i\in
    \Tpoly{}{0}(\Rf)$ and $n>1$;
  \item $\mathscr{K}_n(X,Y_2,\ldots, Y_n)=0$ for all
    $Y_i\in\Tpoly{}{}(\Rf)$ and $n\geq 2$ whenever
    $X\in\Tpoly{}{0}(\Rf)$ is induced by the action of
    $\mathfrak{gl}(d,\R)$.
  \end{enumerate}
\end{theorem}
%Note in particular that, by definition, if $h\in
%\Dpoly{}{0}$ then $\partial h=-\overline{\Delta}(h)$, i.e. the reduced
%coproduct of the augmented coalgebra $\Dpoly{}{0}$.

%%%%%%%%%%%%%%%%%%%%%%%%%%%%%%%%%%%%%%%%%%%%%%%%%%%%%%%%%%%%%%%%%%%%%%%%%%%%%%%%%%%%%%%%%%%%%%%%%%%%%%%%%

%
%Dolgushev-Fedosov-Kontsevich formality
%
\section{Formality for Lie algebroids}
\label{sec:formality}

In this section we recall the formality theorem for Lie
algebroids, which is due to Calaque, see \cite{Calaque2005}. The proof
of this theorem follows the lines of Dolgushev's construction
\cite{Dolgushev2005a, Dolgushev2005} of the
$L_\infty$-quasi-isomorphism from polyvectorfields to polydifferential
operators. The main ingredients are Fedosov's methods
\cite{fedosov:1994a} concerning formal geometry, Kontsevich's
quasi-isomorphism \cite{kontsevich:2003a} and the twisting procedure
inspired by Quillen \cite{Quillen69} (although we use Dolgushev's
version \cite{Dolgushev2005}).  Since we only need
the result and not in fact the details of the construction we are 
rather brief here and refer to \cite{Calaque2005} for details.

\subsection{Fedosov resolutions}

As a first step, Calaque constructs
Fedosov resolutions of polyvector fields and polydifferential operators
of Lie algebroids.

Let us recall that a Lie algebroid is a vector bundle $E \to M$ over a
manifold $M$, equipped with a Lie bracket on sections $\Secinfty (E)$
and an anchor map $\rho \colon E \to TM$, preserving the Lie bracket,
such that
\begin{equation}
  [v, f w]_E
  =
  f[v, w]_E + (\rho(v)f)w,
\end{equation}
for any $v, w \in \Secinfty (E)$ and $f\in \Cinfty (M)$.
Equivalently, we can consider the algebra of $E$-differential forms
$\Secinfty(\wedge^\bullet E^*)$ endowed with the differential $\D_E$
given by $(\D_Ef)(v)=\rho(v)(f)$ for $f\in\Cinfty(M)$ and $X\in
\Secinfty(E)$, by
$(\D_E\alpha)(v,w)=\rho(v)(\alpha(w))-\rho(w)(\alpha(v))-\alpha([v,w]_E)$
for $X,Y\in \Secinfty(E)$ and $\alpha\in\Secinfty(E^*)$ and extended
as a derivation for the wedge product.

The definitions of the DGLA's $\Tpoly{}{}(\Rf)$ and $\Dpoly{}{}(\Rf)$
given in Section~\ref{sec:LocalFormality}
go through mutatis mutandis to define the DGLA's $ \Tpoly{E}{}(M)$ and
$ \Dpoly{E}{}(M)$ starting from the Lie-Rinehart pair
$\left(\Cinfty(M), \Secinfty(E)\right)$. Notice that the resulting
spaces $ \Dpoly{E}{k}(M)$ can be identified with the spaces of
$E$-polydifferential operators of order $k+1$.
In order to extend the result of Theorem~\ref{thm:kontsevich} to any
Lie algebroid, we need to consider the so-called \emph{Fedosov
  resolutions}.  The idea (coming from formal
geometry) consists in replacing the DGLA's $\Tpoly{}{}(\Rf)$ and
$\Dpoly{}{}(\Rf)$ by quasi-isomorphic DGLA's (using DGLA morphisms in
this case). For the rest of this section we consider a Lie algebroid
$E$ of rank $d$.
We denote by $ ^E\FibT{}$ the bundle of formal fiberwise $E$-polyvector
fields over $M$, this is the bundle 
associated to the principal bundle of general linear frames in
$E$ with fiber $\Tpoly{}{}(\Rf)$.
Similarly, the bundle $ ^E\FibD{}$ of formal fiberwise
$E$-polydifferential operators is the bundle over $M$ associated to the
principal bundle of general linear frames in $E$ with fiber
$\Dpoly{}{}(\Rf)$.
The Fedosov resolutions are given on the level of vector spaces by the
$E$-differential forms with values in 
$^E\FibT{}$ and $^E\FibD{}$ respectively. We denote these spaces
by $\FormsT{}{M}{}$ and $\FormsD{}{M}{}$ respectively. Note that these
spaces carry a natural DGLA structure, namely the one induced by the
structure on fibers (which is $\group{GL}(d,\R)$-equivariant).
\begin{lemma}
  \label{explicit}
  There exist $\group{GL}(d,\R)$-equivariant
  isomorphisms of algebras
  \begin{equation}
    \Cinfty(\Rf)\simeq \prod_{k=0}^\infty \Sym^kT^*_0\R^d\simeq \R\Schouten{\hat{x}_1,\ldots, \hat{x}_d}
  \end{equation}
  \begin{equation}
    \Tpoly{}{0}(\Rf)\simeq \Cinfty(\Rf)\otimes T_0\R^d.
  \end{equation}
  Here the Lie algebra structure on $\Cinfty(\Rf)\otimes T_0\R^d$ is
  induced from the action of $T_0\R^d$ as derivations at $0$ on
  $\Cinfty(\R^d)$.
\end{lemma}
The proof of the above lemma can be found in \cite[Prop. 2.1.10]{Niek} and \cite[Theorem
  1.1.3]{MoerdijkReyes}. It implies that
\begin{equation}
  \FormsT{}{M}{}
  \simeq 
  \Secinfty\left(\prod_{k=0}^\infty \Anti^\bullet E\otimes \Sym^k E^*\otimes \Anti^\bullet E^*\right)
\end{equation}
and similarly 
\begin{equation}
  \FormsD{}{M}{}
  \simeq 
  \Secinfty\left(\prod_{k=0}^\infty\left(\bigoplus_{l=0}^\infty \Sym^l E\right)^{\otimes\bullet}\otimes \Sym^k E^*\otimes\Anti^\bullet E^*\right),
\end{equation} 
where $\Anti$ and $\Sym$
denote the anti-symmetric and symmetric algebra, respectively.

The next step consists in finding a differential on the Fedosov
resolutions which is compatible with the graded Lie algebra structure
and which makes them into DGLA's quasi-isomorphic to $\Tpoly{E}{}(M)$
and $\Dpoly{E}{}(M)$, respectively.
Given some trivializing coordinate neighborhood $V\subset M$ of $E$, a
local frame $e_1, \dots, e_d$ and its dual frame $(x_1,\ldots, x_d)$,
we can define the operators
\begin{equation}
  \delta\colon \FormsT{}{V}{}\rightarrow \FormsT{}{V}{}
\end{equation} 
by the formula 
\begin{equation}
  \label{delta}\delta(Y)
  =
  \sum_{i=1}^d x_i\wedge\Schouten{{e_i},Y},
\end{equation} 
In other words $\delta=\Schouten{A_{-1},\argument}$ where
$A_{-1}\in\FormsT{1}{V}{0}$ is the one-form
$A_{-1}=\sum_{i=1}^dx_i\otimes e_i$.  One easily checks that $\delta$
does not depend on the choice of coordinates, since $A_{-1}$ is
independent of coordinates, and therefore extends to all of $M$.
By replacing $\Schouten{\argument,\argument}$ by
$[\argument,\argument]_G$ in \eqref{delta} we obtain the operators
\begin{equation}
  \delta\colon \FormsD{}{M}{}\longrightarrow\FormsD{}{M}{}.
\end{equation}
Note that, since $\Schouten{A_{-1},A_{-1}}=[A_{-1},A_{-1}]_G = 0$, we
have $\delta^2 = 0$. Furthermore, since it is given by an inner
derivation and $\delta\mu = 0$, $\delta$ is compatible with the
fiberwise Lie structures and thus yields DGLA structures.

The cohomology of the complexes $\left( \FormsT{l}{M}{},\delta\right)$
and $\left( \FormsD{l}{M}{},\delta\right)$ is given by the following
proposition (proved e.g in \cite[Prop.~2.1]{Calaque2005}).
\begin{proposition}
  We have that 
  \begin{equation}
    \Cohom{0}( \FormsT{}{M}{},\delta)
    \simeq
    \Gamma^\infty(\Anti^\bullet E)
    \hspace{0.3cm} \mbox{and}\hspace{0.3cm}
    \Cohom{0}( \FormsD{}{M}{},\delta)
    \simeq
    \Secinfty\left(\left(\bigoplus_{l=0}^\infty S^lE\right)^{\otimes\bullet}\right)
  \end{equation} 
  while 
  \begin{equation}
    \Cohom{>0}( \FormsT{}{M}{},\delta)=0
    \hspace{0.3cm}\mbox{and}\hspace{0.3cm}
    \Cohom{>0}( \FormsD{}{M}{},\delta)=0
  \end{equation}
\end{proposition}
Notice that
\begin{equation}
  \Cohom{0}( \FormsT{}{M}{},\delta)
  \simeq \hspace*{0.05cm}
   \Tpoly{E}{}(M)
  \hspace{0.3cm}\mbox{and}\hspace{0.3cm}
   \Cohom{0}( \FormsD{}{M}{},\delta)
  \simeq \hspace*{0.05cm}
   \Dpoly{E}{}(M)
\end{equation}
as vector spaces. This does not provide us with the quasi-isomorphisms
we are looking for, since $\Cohom{0}( \FormsT{}{M}{},\delta)$ carries
the trivial Lie algebra structure. To correct it, the idea is to
construct a perturbation of the differential $\delta$ that does not
affect the size of the cohomology, but only the Lie
algebra structure on cohomology.
Notice that the operator $\delta$ is of degree $-1$ in
terms of the filtration and so we may start perturbing at order $0$,
i.e. adding a connection in the bundle $E$.  The fact that the resulting
perturbation should square to zero forces us to choose a torsion-free
connection $\conn$.  This gives us the operators
\begin{equation}
  \conn
  \colon 
  \FormsT{}{M}{} \longrightarrow \FormsT{}{M}{}
  \hspace{0.3cm}\mbox{and}\hspace{0.3cm} 
  \conn
  \colon 
  \FormsD{}{M}{}\longrightarrow\FormsD{}{M}{}.
\end{equation} 
Thus we consider the corresponding operators $-\delta+\conn$ and
$-\delta+ \conn +\partial$. This leads to the problem that there
is no reason to assume that we can find $\conn$ such that $\conn^2=0$
(since not every Lie algebroid is flat). Following the idea of
Fedosov, we correct $-\delta + \conn$ by an inner derivation and make
the ansatz
\begin{equation}
  D
  :=
  -\delta + \conn + [A, \argument]
\end{equation}
with $A \in \FormsT{1}{M}{0}\hookrightarrow \FormsD{1}{M}{0}$ and
where $[A,\argument]$ means $\Schouten{A,\argument}$ or
$[A,\argument]_G$ depending on the situation.  The trick is to find
$A$ such that $D^2 = 0$, as proved in \cite[Prop.~2.2]{Calaque2005}.
\begin{lemma}
\label{lem:A}
  There exists a unique $A$ such that
  \begin{lemmalist}
   \item $\delta A = R + \conn A + \frac{1}{2} [A, A]$
   \item $\delta ^{-1} A = 0$.
  \end{lemmalist}
\end{lemma}
Here $R$ denotes the curvature of $\conn$ expressed in terms of the
bundle $^E\FibT{}(M)$ (or $ ^E\FibD{}(M)$), i.e. it is given by the
equation
\begin{equation}
  \conn^2Y
  =
  [R,Y]
\end{equation}
and $\delta^{-1}$ is a particular $\delta$-homotopy from the
projection onto degree $0$, denoted $\sigma$, to the identity, i.e.
\begin{equation}
  \delta^{-1}\delta + \delta\delta^{-1} + \sigma
  =
  \id.
\end{equation} 
The condition $\delta^{-1}A = 0$ is simply a normalization condition
ensuring uniqueness of the solution.
\begin{proposition}
  \label{prop:DcohomM}
  We have  
  \begin{equation}
    \Cohom{>0}( \FormsT{}{M}{},D)
    =
    0
    \qquad \mbox{and} \qquad
    %\hspace{0.3cm}\mbox{and}\hspace{0.3cm} 
    \Cohom{>0}( \FormsD{}{M}{},D)
    =
    0.
  \end{equation} 
  Furthermore we have 
  \begin{equation}
  \begin{gathered}
    \Cohom{0}( \FormsT{}{M}{},D)\simeq \Cohom{0}( \FormsT{}{M}{},\delta)
    \\
    \Cohom{0}( \FormsD{}{M}{},D)\simeq \Cohom{0}( \FormsD{}{M}{},\delta).
  \end{gathered}
  \end{equation}
\end{proposition}
\begin{proof}
  \cite[Thm.~2.3]{Calaque2005}
\end{proof}
Let us denote the isomorphisms from the above Proposition by $\tau$.
Then, using a Poincar\'e-Birkhoff-Witt-type isomorphism, Calaque constructs an
isomorphism (see \cite[Sec.~2.3]{Calaque2005})
%\textcolor{red}{I am not sure, he cites Rinehart...}
\begin{equation}
  \nu 
  \colon 
  \Ker\delta\cap\FormsD{0}{M}{}\longrightarrow \Dpoly{E}{}(M)
\end{equation}
of filtered vector spaces. 

Similarly, but in an easier way, we obtain an isomorphism 
\begin{equation}
  \nu
  \colon
  \Ker\delta\cap \FormsT{0}{M}{}\longrightarrow \Tpoly{E}{}(M)
\end{equation}
of graded vector spaces.
Finally, as proved in \cite[Prop.~2.4-2.5]{Calaque2005}, we get:
\begin{theorem}[Fedosov Resolutions]
  \label{thm:FedRes}
  The maps 
  \begin{equation}
    \lambda_D
    \colon 
    \left(\Dpoly{E}{}(M),\partial\right)
    \longrightarrow
    \left(\FormsD{}{M}{},\partial + D\right)
  \end{equation} 
  and 
  \begin{equation}
    \lambda_T
    \colon 
    \left(\Tpoly{E}{}(M),0\right)
    \longrightarrow 
    \left(\FormsT{}{M}{},D\right)
  \end{equation} 
  both given by $\tau\circ \nu^{-1}$  are DGLA quasi-isomorphisms.
\end{theorem}
Let us sketch the remaining steps necessary to obtain the
$L_\infty$-quasi-isomorphisms from $\Tpoly{E}{}(M)$ to $\Dpoly{E}{}(M)$.
As a second step, one notices that in a trivializing neighborhood
$U\subset M$ of
$E$ the connection $\conn$ on both $\FormsT{}{M}{}$ and
$\FormsD{}{M}{}$ is given by $\D_E + [B_U,\argument]$ for some element $B_U\in \FormsT{1}{M}{0}\hookrightarrow\FormsD{1}{M}{0}$. 
Thus, in this neighborhood, we have
$D = \D_E + [\Gamma,\argument]$, where $\Gamma$ is a Maurer-Cartan element.  
We now observe that
the map
\begin{equation}
  \mathscr{U}
  \colon 
  \FormsT{}{U}{}
  \longrightarrow 
  \FormsD{}{U}{}
\end{equation} 
given by applying the map $\mathscr{K}$ from
Theorem~\ref{thm:kontsevich} fiberwise
commutes with $\D_E$. 
The next step consists in twisting this map by
$\Gamma$ to obtain $L_\infty$-quasi-isomorphisms
\begin{equation}
  \mathscr{U}^\Gamma\circ\lambda_T
  \colon \Tpoly{E}{}(U)
  \longrightarrow
  \FormsD{}{U}{}.
\end{equation}
The twisting procedure is essential in our paper and will be discussed
in full detail in Section~\ref{sec:Twisting}.  By using the properties
of Kontsevich's quasi-isomorphism \eqref{eq:kontsevich} and the fact
that $\conn$ is a $\mathfrak{gl}(d,\R)$ connection we find that these
quasi-isomorphisms coincide on intersections and thus we obtain
\begin{equation}
  \mathscr{U}^\Gamma\circ\lambda_T
  \colon 
  \Tpoly{E}{}(M)\longrightarrow\FormsD{}{M}{}.
\end{equation}

\begin{remark} 
  Although it may seem that we are being sloppy with notation by
  writing $\mathscr{U}^\Gamma$, since it is not a twist a priori, it
  is still possible to consider it as a twist in the context of curved
  $L_\infty$-algebras. This construction will be discussed in the
  upcoming paper \cite{sisters:2017b}.
\end{remark}
Finally we would like to define the $L_\infty$-quasi-isomorphism $
\lambda_D^{-1}\circ\mathscr{U}^\Gamma\circ\lambda_T\colon
\Tpoly{E}{}\longrightarrow \Dpoly{E}{}.  $ One problem remains and it
is that, although $\lambda_D$ is obviously injective, we cannot be
assured that $\mathscr{U}^\Gamma\circ\lambda_T$ maps $\Tpoly{E}{}$
into the image of $\lambda_D$. However, Dolgushev
\cite[Prop.~5]{Dolgushev2005a} shows that we can always 
modify $\mathscr{U}^\Gamma\circ\lambda_T$ using a so-called
\emph{partial homotopy} to obtain a new quasi-isomorphism
$\overline{\mathscr{U}}$ which maps into the image of
$\lambda_D$. Thus we obtain the $L_\infty$-quasi-isomorphism
\begin{equation}
  \label{eq:Q}
  F_E
  :=
  \lambda_D^{-1}\circ\overline{\mathscr{U}}\colon \Tpoly{E}{}\longrightarrow \Dpoly{E}{}.
\end{equation} 
As a consequence, we obtain the formality theorem for a generic
manifold $M$ by considering the case $E = TM$ and formality for Lie
algebras by considering the case $E = \lie g$ over a point.
\begin{remark}
  \label{rem:choice}
  Note that the constructions of $D$, $\tau$ and so on are not unique,
  but they depend only on the choice of the torsion-free
  $E$-connection $\conn$.
\end{remark}

%%%%%%%%%%%%%%%%%%%%%%%%%%%%%%%%%%%%%%%%%%%%%%%%%%%%%%%%%%%%%%%%%%%%%%%%%%%%%%%%%%%%%%%%%%%%%%%%%%%%%%%%%%%%%%%%%%%%%%%%%%
%%%%%%%%%%%%%%%%%%%%%%%%%%%%%%%%%%%%%%%%%%%%%%%%%%%%%%%%%%%%%%%%%%%%%%%%%%%%%%%%%%%%%%%%%%%%%%%%%%%%%%%%%%%%%%%%%%%%%%%%%%
%%%%%%%%%%%%%%%%%%%%%%%%%%%%%%%%%%%%%%%%%%%%%%%%%%%%%%%%%%%%%%%%%%%%%%%%%%%%%%%%%%%%%%%%%%%%%%%%%%%%%%%%%%%%%%%%%%%%%%%%%%

\subsection{Twisting procedure}
\label{sec:Twisting}

In the following we recall the notions of twisting DGLA's and
$L_\infty$-morphisms by Maurer--Cartan elements.  The idea of such
twisting procedures comes from Quillen's seminal work
\cite{Quillen69}. Here we follow Dolgushev's approach as
laid out in \cite{Dolgushev2005}. As an example we show how one
obtains the local $L_\infty$-quasi-isomorphisms $\mathscr{U}^\Gamma$
mentioned above.  
\begin{lemma}
\label{lem:exppi}
  Suppose $\pi\in\mathcal{F}^1\mathfrak{L}[1]^0$, then the element
  \begin{equation}
    \exp(\pi):=\sum_{n=0}^\infty\frac{\pi^k}{k!}
  \end{equation}
  is well-defined, invertible and group-like. 
\end{lemma}
\begin{proof} 
  $\exp(\pi)$ is well-defined, since the
  partial sums converge by virtue of $\pi$ being in the first
  filtration (the filtration is respected by $\wedge$).
  Invertibility follows from the usual direct computations showing
  that $\exp(-\pi)\exp(\pi)=1=\exp(\pi)\exp(-\pi)$.  The fact that
  $\exp(\pi)$ is group-like can similarly be deduced from a direct
  computation using the definition of $\Delta$ given in Section~\ref{sec:Linfty}.
\end{proof}
Given $\pi\in\mathcal{F}^1\mathfrak{L}[1]^0$ we define the
$\pi$-twist of the $L_\infty$-algebra $(\mathfrak{L},Q)$ as the
$L_\infty$-algebra $\mathfrak{L}^\pi$ given by the pair
$(\mathfrak{L}, Q^\pi)$ with
\begin{equation}
  Q^\pi(a)
  :=
  \exp(-\pi)\wedge Q(\exp(\pi)\wedge a).
 \end{equation}
\begin{corollary}
Suppose $(\mathfrak{L},Q)$ is an $L_\infty$-algebra and $\pi\in \mathcal{F}^1\mathfrak{L}[1]^0$, then the $\pi$-twist $(\mathfrak{L}, Q^\pi)$ is an $L_\infty$-algebra.
\end{corollary}
\begin{example} 
  Given a curved Lie algebra $(\mathfrak{L}, R,\D,
  [\argument,\argument])$ we find the twisted curved Lie algebra
  $(\mathfrak{L}, R^\pi,
  \D + [\pi,\argument],[\argument,\argument])$, where
  \begin{equation}
  \label{eq:Rpi}
    R^\pi
    :=
    R+\D\pi + \frac{1}{2}[\pi,\pi].
  \end{equation}
  Note in particular that the $\pi$-twist is flat exactly when $\pi$ satisfies the Maurer-Cartan equation.
\end{example}

\begin{proposition}
  \label{prop:MCtwist} 
  Suppose $\mathfrak{L}$ is an $L_\infty$-algebra and $\pi\in\mathcal{F}^1\mathfrak{L}[1]^0$, then the $\pi$-twist $\mathfrak{L}$ is flat if and only if $\pi$ is a Maurer-Cartan element. 
\end{proposition}
\begin{proof} 
  We have 
  \begin{equation} 
    Q^\pi(1)
    =
    \exp(-\pi)\wedge Q(\exp(\pi))
    =
    \sum_{n=0}^\infty \frac{Q_n(\pi^n)}{n!}
    ,
  \end{equation}
  since all terms in $\bigoplus_{k=2}^\infty\Anti^k(\mathfrak{L}[1])$
  cancel out by virtue of the fact that
  $Q^\pi(1)=Q^\pi_0(1)\in\mathfrak{L}[1]$.
\end{proof}
\begin{example} 
  For a DGLA $(\mathfrak{L}, \D,[\argument,\argument])$ Eq.~\eqref{eq:Rpi}
  boils down to the usual Maurer--Cartan equation
  \begin{equation}
    \D\pi+\frac{1}{2}[\pi,\pi]
    =
    0.
  \end{equation}
  If we have similarly a curved Lie algebra with curvature $-R$ it
  comes down to the non-homogeneous equation
  \begin{equation}
    \D\pi+\frac{1}{2}[\pi,\pi]
    =
    R.
  \end{equation}
\end{example}
\begin{lemma}
  \label{lem:MCcondition}
  Suppose $\pi\in\mathcal{F}^1\mathfrak{L}[1]^0$, then $\pi$ is an MC
  element if and only if $Q(\exp(\pi))=0$.
\end{lemma}
\begin{proof}
  The proof follows from the following equation
  \begin{align*}
    Q(\exp(\pi))
    &=
    \sum_{n=0}^\infty\sum_{k=0}^n\sum_{\sigma\in \mbox{\tiny Sh($k$,$n-k$)}}\epsilon(\sigma)\frac{1}{n!}Q_k(\pi^k)\wedge
    \pi^{n-k}
    \\
    &=
    \sum_{n=0}^\infty\sum_{k=0}^n\frac{1}{k!(n-k)!}Q_k(\pi^k)
    \wedge\pi^{n-k}
    \\
    &=
    \left(\sum_{n=0}^\infty \frac{Q_n(\pi^n)}{n!}\right)
    \wedge\exp(\pi).
  \end{align*}
\end{proof}

\begin{lemma}
  \label{lem:Fep=epF}
  Given an L$_\infty$-morphism $F$ from $\mathfrak{L}$ to
  $\mathfrak{L}'$ and an element
  $\pi\in\mathcal{F}^1\mathfrak{L}[1]^0$, we define the $F$-associated
  element $\pi_F\in\mathcal{F}^1\mathfrak{L}'[1]^0$ by the formula
  \begin{equation}
    \pi_F
    :=
    \sum_{n=1}^\infty\frac{F_n(\pi^n)}{n!}.
  \end{equation}
  We have
  \begin{equation}
    F(\exp(\pi))=\exp(\pi_F)
  \end{equation}
\end{lemma}
\begin{proof} 
  It follows from explicit computation using the formula
  \eqref{eq:coalgebramorphism}.
%\[F(\sum_{n=0}^\infty\frac{\pi^n}{n!})=1+
%\sum_{n=1}^\infty\sum_{p=1}^\infty\sum_{\substack{1\leq k_1,\ldots,k_p\leq n\\ k_1+\ldots+k_p=n}}\sum_{\sigma\in\mbox{\tiny Sh($k_1,\ldots, k_p$)}}
%\frac{\epsilon(\sigma)}{n!p!}F_{k_1}(\pi^{k_1})\wedge\ldots\wedge F_{k_p}(\pi^{k_p})=\]
%\[1+\sum_{n,p=1}^\infty\sum_{\substack{1\leq k_1\leq\ldots\leq k_p\leq n\\ k_1+\ldots+k_p=n}}
%\frac{1}{k_1!\ldots k_p!}F_{k_1}(\pi^{k_1})\wedge \ldots\wedge F_{k_p}(\pi^{k_p})=\]
%\[\sum_{n=0}^\infty\frac{1}{n!}\left(\sum_{k=1}^\infty\frac{F_k(\pi^k)}{k!}\right)^n=\exp(\pi_F)\]
\end{proof}
Lemmas \ref{lem:Fep=epF} and \ref{lem:MCcondition} imply the following corollary.
\begin{corollary}
   If $\pi$ is an MC element, then $\pi_F$ is also an MC element.
\end{corollary}
Let $F\colon (\mathfrak{L},Q)\longrightarrow (\widetilde{\mathfrak{L}},\widetilde{Q})$
be an $L_\infty$-morphism and $\pi\in\mathcal{F}^1\mathfrak{L}[1]^0$. 
\begin{definition}[$\pi$-twist morphism]
  The $\pi$-twist of $F$ is a map
  \begin{equation}
    F^\pi
    \colon 
    (\mathfrak{L},Q^\pi)\longrightarrow (\widetilde{\mathfrak{L}},\widetilde{Q}^\pi)
  \end{equation}
  defined by 
  \begin{equation}
    F^\pi(a)
    :=
    \exp(\pi_F)\wedge F(\exp(\pi)\wedge a).
  \end{equation}
\end{definition}
\begin{corollary}
  The $\pi$-twist of an $L_\infty$-morphism $F$ is an
  $L_\infty$-morphism.
\end{corollary}
\begin{proof}
  Note that, by Lemma~\ref{lem:exppi} and Remark~\ref{rem:bialgebra},
  the operators of multiplication by $\exp(\pi)$ and $\exp(\pi_F)$ are
  co-algebra morphisms. Thus $F^\pi$ is a co-algebra morphism. The
  relation $F^\pi Q^\pi=\widetilde{Q}^\pi F^\pi$ follows from
  the definitions.
\end{proof}
\begin{remark}
  \label{comptwist}
  Given two $L_\infty$-morphisms $F$ and $G$ from $\mathfrak{L}$
  to $\mathfrak{L}'$ and $\mathfrak{L}'$ to $\mathfrak{L}''$,
  respectively, and the elements $\pi, B\in
  \mathcal{F}^1\mathfrak{L}[1]^0$, we have that
  \begin{align} 
    (Q^{\pi})^B
    &=
    Q^{\pi+B}=(Q^B)^\pi,
    \\ 
    (F^\pi)^B
    &=
    F^{\pi+B}=(F^B)^\pi,
    \\
    \pi_F+B_{F^\pi}
    &=(\pi+B)_F
    =
    B_F+\pi_{F^B},
    \\
    (\pi_F)_G
    &=
    \pi_{G\circ F}.
  \end{align}
\end{remark}
For the proof of the following proposition we refer to
\cite[Prop. 1]{Dolgushev2005}.
\begin{proposition}
  \label{prop:twistquis}
  Let $F\colon\mathfrak{L}\rightarrow\widetilde{\mathfrak{L}}$ be
  an $L_\infty$-quasi-isomorphism such that the induced morphisms
  \begin{equation}
    F|_{\mathcal{F}^k\mathfrak{L}}
    \colon 
    \mathcal{F}^k\mathfrak{L}\longrightarrow \mathcal{F}^k\widetilde{\mathfrak{L}}
  \end{equation}
  are also $L_\infty$-quasi-isomorphisms for all $k$. Suppose further
  that $\pi\in\mathcal{F}^1\mathfrak{L}[1]^0$ is an MC element. Then
  the $\pi$-twist
  \begin{equation}
    F^\pi\colon \mathfrak{L}^\pi\longrightarrow \widetilde{\mathfrak{L}}^{\pi_F}
  \end{equation}
  of $F$ is also a quasi-isomorphism. 
\end{proposition}
\begin{remark}
  \label{rem:thepoint}
  The proposition above says that the class of
  $L_\infty$-quasi-isomorphisms is closed under the operation of
  twisting by a Maurer--Cartan element. This provides the method of
 showing that an $L_\infty$-morphism is an $L_\infty$-quasi-isomorphism by showing
 that it is the twist of a known $L_\infty$-quasi-isomorphism.
\end{remark}
\begin{example}[Formality for $\R^d$]
  \label{ex:formalityRd}
  Here we generalize the result of Theorem \ref{thm:kontsevich} from
  $\Rf$ to $\R^d$ by providing an example of the claim in
  Remark~\ref{rem:thepoint}. From now on we set $E = TM$ and drop the
  $E$ for notational convenience. Proposition \ref{prop:twistquis}
  allows us to obtain an $L_\infty$-quasi-isomorphism witnessing the
  formality of $\Dpoly{}{}(M)$ for any manifold by twisting the formal
  quasi-isomorphism of Theorem~\ref{thm:kontsevich}.
  We set $M=\R^d$
  %although any manifold
  %allowing for a flat and torsion-free linear connection would do.
  and recall that we are looking for an $L_\infty$-quasi-isomorphism
  \begin{equation}
    \mathscr{U}^\delta
    \colon
    (\Omega(\R^d;\FibT{}),D)
    \longrightarrow 
    (\Omega(\R^d,\FibD{}),\partial+D),
  \end{equation}
  since this would complete 
  the diagram 
  \begin{equation}
  \label{eq:seqquisRd}
    (\Tpoly{}{}(\R^d),0)
    \stackrel{\lambda_T}{\longrightarrow} 
    (\Omega(\R^d;\FibT{}),D)
    \stackrel{\mathscr{U}^\delta}{\longrightarrow} 
    (\Omega(\R^d,\FibD{}),\partial+D)
    \stackrel{\lambda_D}{\longleftarrow} (\Dpoly{}{}(\R^d),\partial)
  \end{equation}
  of $L_\infty$-quasi-isomorphisms. 
  %Here we have suppressed and in the
  %future we will suppress the mention of brackets for notational
  %convenience. 
  Also, recall that $D:=-\delta+\D$. We obtain this map
  $\mathscr{U}^\delta$ as follows.
  First we note that, by applying the map $\mathscr{K}$ from 
  Theorem~\ref{thm:kontsevich} fiberwise, we obtain the $L_\infty$-morphism
  \begin{equation}
    \mathscr{U}
    \colon 
    (\Omega(\R^d;\FibT{}),\D)
    \longrightarrow 
    (\Omega(\R^d;\FibD{}),\partial+\D).
  \end{equation}
  By considering the filtrations by exterior degree on both these
  algebras we construct spectral sequences which show that
  $\mathscr{U}$ is a quasi-isomorphism. Using this same filtration we
  may consider the MC element
  $-A_{-1}\in\mathcal{F}^1\Omega(\R^d;\FibT{})$. Now note that
  $\Omega(\R^d;\FibT{})^{-A_{-1}}$ is exactly
  $(\Omega(\R^d;\FibT{}),D)$ and
  $\Omega(\R^d;\FibD{})^{(-A_{-1})_\mathscr{U}}$ is exactly
  $(\Omega(\R^d;\FibD{}),D)$, since $(-A_{-1})_{\mathscr{U}}=-A_{-1}$
  by point \refitem{thm:kontsevich2} of Theorem~\ref{thm:kontsevich}. So we obtain the
  diagram \eqref{eq:seqquisRd} by setting $\mathscr{U}^\delta:=
  \mathscr{U}^{-A_{-1}}$. This concludes the example of the claim in
  Remark~\ref{rem:thepoint}.
  In order to obtain the quasi-isomorphism 
  \begin{equation}
    \Tpoly{}{}(\R^d)\longrightarrow \Dpoly{}{}(\R^d)
  \end{equation} 
  we need to invert the final arrow of diagram \eqref{eq:seqquisRd}.
  This arrow is actually an identification (by
  DGLA-morphism) with the kernel of $D$ in exterior degree $0$.  Thus
  it can be inverted without problems if we can guarantee that the map
  $\mathscr{U}^\delta\circ\lambda_T$ maps $\Tpoly{}{}(\R^d)$ into this
  kernel.  We refer to \cite[Sect.~4.2]{Dolgushev2005a} for an
  explanation of a way to correct $\mathscr{U}^\delta$ to have this
  property.
\end{example}
% \begin{remark} 
%   Note that there is nothing special about the
%   Maurer-Cartan element $-A_{-1}$ used in the previous construction on
%   $\R^d$.  So, if we consider some arbitrary manifold $M$ with
%   corresponding Fedosov connection $D$ induced by the linear
%   connection $\conn$ then we proceed as follows.  We note that there
%   is a cover of coordinate neighborhoods $U$ such that
%   $\conn|_U=d+[B_U,\argument]$ for some $B_U$ in
%   $\FormsT{1}{U}{0}$. Thus $D|_U=d+[B_u-A_{-1}+A,\argument]$ and we
%   find that, by definition of $A$ (see lemma \ref{lem:A}), $B_U-A_{-1}+A$
%   is an MC element. Thus we obtain the local
%   $L_\infty$-quasi-isomorphism $\mathscr{U}^\Gamma$ from
%   $\FormsT{}{U}{}$ to $\FormsD{}{U}{}$, where
%   $\Gamma=B_U-A_{-1}+A$. It turns out that these morphisms agree on
%   intersections, which finishes the sketch of the proof of the
%   formality theorem.
% \end{remark}

\begin{example}[Formality for Lie algebras]
  Let us conclude this section by providing the equivalent of the
  proof of formality for the case where $M=\{\mbox{pt}\}$ is the
  connected $0$-dimensional manifold and $E$ is a $d$-dimensional Lie
  algebra $\mathfrak{g}$. The DGLA of polyvector fields is given by
  $\mathrm{CE}_\bullet(\mathfrak{g})$, the Chevalley-Eilenberg complex
  with the trivial differential. The complex of $E$-differential forms
  with values in the fiberwise polyvector fields is thus given by
  \begin{equation}
    \mathrm{CE}^\bullet(\mathfrak{g};\mathrm{CE}_\bullet(\mathfrak{g};
    \widehat{\mathcal{S}}(\mathfrak{g}^*))),
  \end{equation}
  where we have denoted $\widehat{\mathcal{S}}(\mathfrak{g}^*)=
  \prod_{k\geq 0}S^k\mathfrak{g}^*$ and the differential
  $\delta-d_E$ coincides with the usual Chevalley-Eilenberg
  differential. A linear $E$-connection $\conn$ is simply given by a
  linear map
  \begin{equation}
    \conn\colon \mathfrak{g}\otimes\mathfrak{g}\longrightarrow \mathfrak{g}.
  \end{equation}
  The corresponding map
  $\conn\colon\Anti^\bullet\mathfrak{g}^*\rightarrow \Anti^{\bullet
    +1}\mathfrak{g}^*$ is given by extending the formula
  \begin{equation}
    \conn\alpha(X, Y)=-\alpha(\conn(\frac{1}{2}(X\otimes Y-Y\otimes X)))
  \end{equation}
  from one-forms as a $\wedge$-derivation. Note that the
  $\mathfrak{g}$-differential $\D_\mathfrak{g}$ is simply given by
  $X\otimes Y\mapsto [X,Y]$. Suppose $\{e_i\}_{i=1}^d$ is a basis for
  $\mathfrak{g}$ with dual basis $\{e^i\}_{i=1}^d$. Then the
  torsion-freeness of the connection $\conn$ can be expressed  as
  $\tilde{\Gamma}_{ij}^k=\tilde{\Gamma}_{ji}^k$ in terms
  of the Christoffel symbols $\tilde{\Gamma}_{ij}^k\in \R$ defined by
  \begin{equation}
    \conn(e_i\otimes e_j)=\tilde{\Gamma}_{ij}^ke_k,
  \end{equation}
  where we have used the Einstein summation convention. Let us consider
  also the \emph{relative} Christoffel symbols
  $\Gamma_{ij}^k$ defined by
  \begin{equation}
    \conn(e_i\otimes e_j)-[e_i,e_j]=\Gamma_{ij}^ke_k,
  \end{equation}
  i.e. $\Gamma_{ij}^k=\tilde{\Gamma}_{ij}^k-\frac{1}{2}c_{ij}^k$ where
  $c_{ij}^k$ are the structure constants. In terms of these torsion-freeness is
  equivalent to the equation
  \begin{equation}
    \Gamma_{ij}^k-\Gamma_{ji}^k-c_{ij}^k=0.
  \end{equation}
  Note in particular that the connection $\D_\mathfrak{g}$ is
  \underline{not} torsion-free. The most obvious choice of
  torsion-free connection is given by
  $\Gamma_{ij}^k=\frac{1}{2}c_{ij}^k$, but we leave the choice of
  symmetric part open. Given any connection $\conn$ it is given on
  $\mathrm{CE}^\bullet(\mathfrak{g};\mathrm{CE}_\bullet(\mathfrak{g};
  \widehat{\mathcal{S}}(\mathfrak{g}^*)))$ by the formula
  \begin{equation}
    \conn=\D_\mathfrak{g}+[\Gamma_{ij}^k e^i\hat{e}^je_k,\argument]
  \end{equation}
  where we have used the hat to signify that we consider $\hat{e}^j\in
  \widehat{S}(\mathfrak{g}^*)$. Similar statements hold for
  $\Dpoly{\mathfrak{g}}{}$.
%In particular we note that the
%differential $\conn$ is obtained as the twist of $\D_\mathfrak{g}$ by 
%$\Gamma=\Gamma_{ij}^ke^i\hat{e}^je_k$ and so we find globally curved
%$L_\infty$-algebras as per Example \ref{Lieglob}. 
  Now the example proceeds identically to the previous one. 
\end{example}

%%%%%%%%%%%%%%%%%%%%%%%%%%%%%%%%%%%%%%%%%%%%%%%%%%%%%%%%%%%%%%%%%%%%%%%%%%%%%%%%%%%%%%%%%%%%%%%%%%%%%%%%%%%%%

%
%Formality and Deformation Symmetries
%

\section{Formality and Deformation Symmetries}
\label{sec:FormalityActions}

In this section we prove the main result of this paper, which 
leads to a new perspective on Drinfeld's approach to deformation
quantization. First we construct certain $L_\infty$-algebras related
to a Hopf algebra or more generally a unital bialgebra and show how
one obtains deformations from Drinfeld twists and maps into a
Hochschild cochain complex. Then we briefly recall the basic notions
of Poisson action and triangular Lie algebra.  We consider the
particular case of a Poisson action of a triangular Lie algebra $(\lie
g, r)$ on a manifold $M$ and we show that we can construct a
corresponding $L_\infty$-morphism between polydifferential operators
$\Dpoly{\mathfrak{g}}{}$ and $\Dpoly{}{}(M)$.  This morphism induces a
DGLA morphism between a quantum group associated to our Lie algebra
and a deformed algebra of smooth functions on $M$.

\subsection{Deformation Symmetries}

In the following we define the concept of a deformation
symmetry. This notion is inspired by Drinfeld's work on deformation
through quantum actions and Drinfeld twists. Let us start by
recalling the definition of Drinfeld twist. In this section we shall
fix the Hopf algebra $(H,\Delta, \epsilon, S)$ over the PID $\ring R$
containing $\Q$.

\begin{definition}[Drinfeld twist, \cite{drinfeld:1983a, drinfeld:1988a}]
  \label{def:TwistUEA}%
  An element $J\in H\otimes H$ is said to be a twist on $H$ if the
  following three conditions are satisfied.
  \begin{definitionlist}
  \item\label{def:twist1} $J$ is invertible;
  \item \label{def:twist2} 
    $(\Delta\tensor 1)(J)(J\tensor 1) = (1\tensor\Delta)(J)(1\tensor J)$
    and
  \item \label{def:twist3}
    $(\epsilon\tensor 1)J =
    (1\tensor\epsilon)J = 1$.
  \end{definitionlist}
\end{definition}
In the following we consider formal deformations. If we consider
twists in $H\hhbar$, the condition of invertibility and
``co-invertibility'' (condition \refitem{def:twist3} in the above definition) 
may be replaced by a stronger condition which is
easier to check. In fact this condition may be formulated for any Hopf
algebra equipped with a complete filtration $H = \mathcal{F}^0 H
\supset \mathcal{F}^1H\supset\ldots $.
\begin{definition}[Formal Drinfeld twist]
  \label{def:FTwistUEA}%
  Let $H$ be equipped with the complete filtration
  \\$H = \mathcal{F}^0 H \supset \mathcal{F}^1 H \supset \ldots$. Then an
  element $J \in H \otimes H$ is said to be a formal twist on $H$
  if $J$ satisfies \refitem{def:twist2} of Definition \ref{def:TwistUEA} and
  $J-1\otimes 1\in \mathcal{F}^1(H\otimes H)$.
\end{definition}
\begin{corollary}
  \label{cor:ft=t}
  A formal twist on $H$ is a twist on $H$.
\end{corollary}
\begin{proof} 
  This follows immediately from the compatibility of the Hopf
  algebra structure with the complete filtration.
\end{proof}
It turns out that the definition of formal twist coincides exactly
with the definition of Maurer-Cartan element on a certain DGLA that we
shall now define. The main observation is that the formulas
\eqref{eq:Pre-Lie} and \eqref{eq:Lie} for the Gerstenhaber bracket on
$\Dpoly{E}{}$ only involve the structure of a unital bialgebra.
From now on we denote 
\begin{equation}
  TH = \bigoplus_{k=0}^\infty T^k H
  \qquad
  \mbox{with}
  \qquad
  T^k H := H^{\otimes k}.
\end{equation}
For $P_1\in T^{k_1+1}H$ and $P_2\in
T^{k_2+1}H$ set
\begin{equation}
  \label{H-Pre-Lie}
  P_1\bullet P_2
  :=
  \sum_{i=0}^{k_1}
  (-1)^{ik_2}(\id^{\otimes i}\otimes\Delta^{(k_2)}\otimes 
  \id^{\otimes k_1-i})(P_1)\cdot (1^{\otimes i}\otimes P_2\otimes 1^{\otimes k_1-i})
\end{equation}
and 
\begin{equation}
  \label{H-Lie}
  [P_1,P_2]_H
  :=
  P_1\bullet P_2-(-1)^{k_1k_2}P_2\bullet P_1
\end{equation}
\begin{proposition} 
  The graded vector space $TH[1]$ equipped with the bracket
  $[\argument,\argument]_H$ is a graded Lie algebra.
\end{proposition}
\begin{proof}

  \leavevmode

  \noindent We can immediately extend $[\argument,\argument]_H$ to
  non-homogeneous elements, since $\bullet$ can be extended by bilinearity.  Thus
  the bilinearity and anti-symmetry of $[\argument,\argument]_H$
  follow immediately from the bilinearity of $\bullet$, which follows
  in turn from the linearity of the coproduct and the bilinearity of
  the product. Finally denote the associator of $\bullet$ by $\alpha$,
  i.e.
  \begin{equation} 
    \alpha(A,B,C)=A\bullet(B\bullet C)- (A\bullet B)\bullet C.
  \end{equation}
  Then the average of $\alpha$ over the symmetric group $S_3$ is $0$, i.e 
  \begin{equation}\label{real-pre-lie}
    \sum_{\sigma\in S_3}\sigma^*\alpha =0.
  \end{equation}
  Here $S_3$ acts on $(TH[1])^{\otimes 3}$ through the usual signed
  permutation of tensor legs. The last equation is obviously
  equivalent to the Jacobi identity for $[\argument,\argument]_H$.
\end{proof}
\begin{remark}
  \label{rem:braces}
  The structure $\bullet$ on $TH[1]$ is
  actually the pre-Lie structure coming from a brace algebra
  structure.  As such the identity \eqref{real-pre-lie} can actually
  be proved by showing the pre-Lie identity
  \begin{equation}
    \alpha(A,B,C)
    =
    (-1)^{|A||B|}\alpha(B,A,C).
  \end{equation}
  The braces underlying the brace algebra structure are given by 
  \begin{align*}
    &P \langle Q_1,\ldots, Q_r\rangle 
    =
    \sum_{0\leq i_1< i_2<\ldots<i_r\leq k}(-1)^{i_1k_1+i_2k_2+\ldots+i_rk_r}
    \left(
    \id^{\otimes i_1}\otimes\Delta^{(k_1)}\otimes \id^{\otimes (i_2-i_1)}\otimes\ldots \right.
    \\
    & \left. \ldots \otimes\Delta^{(k_r)}\otimes \id^{\otimes (k-i_r)}
    \right)(P)\cdot 
    \left(1^{\otimes i_1}\otimes Q_1\otimes 1^{\otimes (i_2-i_1)}\otimes\ldots\otimes 1^{\otimes (i_{r}-i_{r-1})}\otimes Q_r\otimes 1^{\otimes (k-i_r)}\right),
  \end{align*}
  where $P\in T^{k+1}H$ and $Q_j\in T^{k_j+1}H$ for all $j$. Note that
  restricting this brace algebra structure to $H=T^1H=(TH[1])^0$ we
  find the brace algebra given in \cite[Section~6]{Aguiar}.
\end{remark}
%
%\begin{definition}
  %We define the DGLA $H^\bullet_{poly}$ as the twist of $TH[1]$ by the
  %Maurer-Cartan element $1\otimes 1$. In other words 
Let us denote the twist of $TH[1]$ by the ``Maurer-Cartan" element $1\otimes 1$ as
\begin{equation}
  \left(H^\bullet_{poly},[\argument,\argument],\partial\right)
  :=
  \left(T^{\bullet +1}H,[\argument,\argument]_H,[1\otimes1,\argument]_H\right).
\end{equation}
%\end{definition}
%
\begin{lemma}
  An element $J\in T^2H$ is a formal twist on $H$ if and only if
  $J-1\otimes 1$ is a Maurer-Cartan element in $H_{poly}$.
\end{lemma}
\begin{proof} 

  \leavevmode

  \noindent Suppose first that $J\in T^2H$ is a formal twist on
  $H$. Then, by definition, $F := J-1\otimes 1$ is an element of
  $\mathcal{F}^1H_{poly}^1$ and
  \begin{equation}
    \partial F+\frac{1}{2}[F,F]_H
    =
    \frac{1}{2}[J,J]_H
    =
    (\Delta\otimes \id)(J)(J\otimes 1)-(\id\otimes \Delta)(J)(1\otimes J)=0
  \end{equation}
  by condition \refitem{def:twist2} of Definition \ref{def:TwistUEA}.  Conversely,
  suppose $F$ is a Maurer-Cartan element in $H_{poly}$, then, for $J =
  1\otimes 1 + F$,
  \begin{equation} 
    (\Delta\otimes \id)(J)(J\otimes 1)-(\id\otimes \Delta)(J)(1\otimes J)
    =
    \frac{1}{2}[J,J]_H
    =
    \partial F +\frac{1}{2}[F,F]_H
    =
    0
  \end{equation}
  by the Maurer-Cartan equation. So $J$ satisfies \refitem{def:twist2} of
  \ref{def:TwistUEA}, while clearly $J-1\otimes
  1=F\in\mathcal{F}^1(H\otimes H)$.
 \end{proof}
Drinfeld discovered \cite{drinfeld:1988a,drinfeld:1983b} that one can
twist the Hopf algebra structure on $H$ by any (formal) twist $J$. 
More explicitely, one obtains the twisted Hopf algebra $H_J$
by changing only the coproduct $\Delta$ to $\Delta_J$ given by
\begin{equation} 
  \Delta_J(X)
  =
  J^{-1}\Delta(X)J
\end{equation} 
Let us fix the (formal) twist $J$ on $H$. It
is convenient to introduce the following notation
\begin{align*}
  J_k
  &:=
  \prod_{i=1}^{k}(\Delta^{(k-i)}\otimes \id^{\otimes i})(J\otimes 1^{\otimes i-1})
  \\
  &=
  (\Delta^{(k-1)}\otimes \id)(J)\cdot(\Delta^{(k-2)}\otimes \id\otimes \id)(J\otimes 1)
  \cdot\ldots\cdot
  (\Delta\otimes \id^{\otimes k-1})(J\otimes 1^{\otimes k-2})\cdot (J\otimes 1^{\otimes k-1})
\end{align*}
and we set $J_0=1$.  Note that $J_k\in H^{\otimes k+1}$ is invertible
with $J_k^{-1}$ given by reversing
the order of terms above and replacing $J$ by $J^{-1}$.
\begin{lemma}
  \label{lem:DeltaJk}
  The iterates of $\Delta_J$ are given by the formula 
  \begin{equation} 
    \Delta_J^{(k)}(X)
    =
    J_k^{-1}\Delta^{(k)}(X)J_k.
  \end{equation}
\end{lemma}
\begin{proof} 
  The proof is given by straightforward computation.  
\end{proof}
We also need the following (slightly technical) lemma.
\begin{lemma}
  \label{lem:tech}
  The $J_k$'s satisfy the relation 
  \begin{equation} 
  \label{eq:tech}
    (\id^{\otimes i}\otimes \Delta^{(l)}\otimes \id^{\otimes k-i})(J_k)
    \cdot
    (1^{\otimes i}\otimes J_l\otimes 1^{\otimes k-i})
    =
    J_{k+l}
  \end{equation}
  for all $0\leq i\leq k\in\Z$ and $l\in \Z_{\geq0}$. 
\end{lemma}
\begin{proof}
  For $k=i=0$, the formula \eqref{eq:tech} reads
  \begin{equation} 
    \Delta^{(l)}(1)\cdot J_l
    =
    J_l
  \end{equation}
  which is obviously satisfied. For $k=1$ and $i=0$ we get
  \begin{equation} 
    (\Delta^{(l)}\otimes \id)(J)\cdot(J_l\otimes 1)
    =
    J_{l+1},
  \end{equation}
  which follows immediately from the definition of $J_{l+1}$. For $k=1$, $i=1$ and $l=0$ we find 
  \begin{equation}
  J\cdot(1\otimes 1)=J.
  \end{equation}
  We will fully establish the $k=1$ case by induction now. So suppose the formula \ref{eq:tech} holds for $k=1$, $i=1$, $l-1\in \Z_{\geq0}$. Then 
  \begin{align*}
    ((\id\otimes \Delta^{(l)})(J)\cdot(1\otimes J_l)
    &= (\id\otimes \Delta^{(l-1)}\otimes\id)\left((\id\otimes \Delta)(J)\cdot(1\otimes J)\right)\cdot(1\otimes J_{l-1}\otimes 1)
    \\
    &=
    (\id\otimes\Delta^{(l-1)}\otimes \id)\left((\Delta\otimes \id)(J)\cdot
    (J\otimes 1)\right)\cdot(1\otimes J_{l-1}\otimes 1)
    \\
    &=
    (\Delta^{(l)}\otimes \id)(J)\cdot(\id\otimes \Delta^{(l-1)}\otimes \id)(J\otimes 1)\cdot (1\otimes J_{l-1}\otimes 1)
    \\
    &=
    (\Delta^{(l)}\otimes \id)(J)\cdot\left(((\id\otimes \Delta^{(l-1)})(J)\cdot(1\otimes J_{l-1}))\otimes 1\right)
    \\
    &=(\Delta^{(l)}\otimes \id)(J)\cdot (J_l\otimes 1)=J_{l+1}.
  \end{align*}
  Thus we have established the $k=0$ and $k=1$ cases completely. 
  %Thus we can establish the $k=1$ case completely by induction if we
  %establish 
 To establish the $k\geq 2$ cases, we proceed by induction. Suppose that the formula \eqref{eq:tech} is
  satisfied for all triples $(\kappa,i,l)\in(\Z_{\geq 0})^3$ where $i\leq \kappa$ and $\kappa<k$. Then we
  have
  \begin{align*} 
    (\id^{\otimes i}\otimes \Delta^{(l)}\otimes \id^{\otimes k-i})(J_k)&\cdot (1^{\otimes i}\otimes J_l\otimes 1^{\otimes k-i})
    =
    (\id^{\otimes i}\otimes \Delta^{(l)}\otimes \id^{k-i})\left((\id\otimes \Delta^{(k-1)})(J)\cdot(1\otimes J_{k-1})\right)\cdot\\ 
    &\hspace{0.4cm}\cdot(1^{\otimes i}\otimes J_l\otimes 1^{\otimes k-i})
    \\
    &=
    (\id^{\otimes i}\otimes \Delta^{(l)}\otimes\id^{\otimes k-i})((\id\otimes \Delta^{(k-1)})(J))\cdot
    \\
    &\hspace{0.4cm}\cdot(1\otimes \left((\id^{\otimes (i-1)}\otimes \Delta^{(l)}\otimes \id^{\otimes k-i})(J_{k-1})\right)\cdot(1^{\otimes i-1}\otimes J_l\otimes 1^{k-i}))
    \\
    &=
    (\id^{\otimes i}\otimes \Delta^{(l)}\otimes \id^{\otimes k-i})((\id\otimes \Delta^{(k-1)})(J))\cdot(1\otimes J_{k+l-1})
    \\
    &=(\id\otimes \Delta^{(k+l-1)})(J)\cdot(1\otimes J_{k+l-1})
    \\
    &=
    J_{k+l}.
  \end{align*}
  Thus, Eq.~\eqref{eq:tech} is satisfied for all triples $(k,i,l)$. 
\end{proof}
\begin{proposition}
  \label{prop:twistJ}
  The map $\mathscr{J}\colon (H_J)_{poly}\longrightarrow
  (H_{poly})^{J-1\otimes 1}$ given by $\mathscr{J}(P)=J_k\cdot P$ for
  $P\in (H_J)_{poly}^k$ is a DGLA isomorphism.
\end{proposition}
\begin{proof} 
  First we note that $\mathscr{J}^{-1}(P)=J_k^{-1}P$ for $P\in
  H_{poly}^k$ is obviously an inverse of $\mathscr{J}$. Furthermore we
  observe that $\mathscr{J}(1\otimes 1) = J$, which shows that we only need
  to check that $\mathscr{J}$ preserves the brackets. This follows by
  direct computation from Lemmas \ref{lem:DeltaJk} and \ref{lem:tech}
\end{proof}
\begin{definition}[Deformation Symmetry]
  \label{Defsym}
  Let $\algebra A$ be a unital associative algebra over $\ring R$
  equipped with a complete filtration $\algebra A =\mathcal{F}^0
  \algebra A\supset\mathcal{F}^1 \algebra A\supset\ldots$ and consider
  the Hochschild DGLA structure on the complex $C^\bullet(\algebra
  A) = C^\bullet(\algebra A; \algebra A)[1]$. A deformation symmetry of
  $H$ in $\algebra A$ is a map
  \begin{equation} 
    \Phi
    \colon 
    H_{poly}\longrightarrow C(\algebra A)
  \end{equation}
  of $L_\infty$-algebras. 
\end{definition}
The previous proposition implies the following claim. 
\begin{corollary}
  Any formal twist on $H$ produces
  deformations of all algebras $\algebra A$ equipped with a
  deformation symmetry of $H$.
\end{corollary}
The notion of deformation symmetry is a generalization of the
standard notion of universal deformation via Drinfeld twist, see e.g.
\cite{bieliavsky.gayral:2015a,ESW2016,giaquinto.zhang:1998a}. Universal deformation formula 
relies on the notion of Hopf
algebra action that we recall in the following definition.
\begin{definition}[Hopf algebra action]
  Let $\algebra A$ be as in Definition \ref{Defsym}. Then an action of
  the Hopf algebra $H$ on $\algebra A$ is defined as a map
  \begin{equation}
    \phi\colon H\otimes \algebra A\longrightarrow \algebra A
  \end{equation}
  such that 
  \begin{align}
    \phi\circ(\id_H\otimes \mu) 
    &= 
    \mu\circ (\phi\otimes\phi)
    \circ
    (\id_H\otimes \tau_{H\otimes \algebra A}\otimes\id_{\algebra A})
    \circ 
    (\Delta\otimes\id_{\algebra A\otimes \algebra A}), 
    \\
    \phi\circ(\mu_H\otimes \id_{\algebra A})
    &=
    \phi\circ(\id_H\otimes \phi)
    \\
    \phi\circ(\eta_H\otimes \id_{\algebra A})
    &=
    \mu_{\algebra A}\circ(\eta_{\algebra A}\otimes \id_{\algebra A})
    \\
    \phi\circ(\id_H\otimes \eta_{\algebra A}) &= \eta_{\algebra A}\circ\mu_{\ring R}\circ(\epsilon\otimes \id_{\ring R}).
  \end{align}
  Here $\mu$ denotes the multiplication of $\algebra A$, $\mu_{\ring R}$
  denotes the multiplication of $\ring R$, $\mu_H$ denotes the
  multiplication of $H$, $\eta_{\algebra A}\colon \ring R\rightarrow
  A$ denotes the unit of $A$, $\eta_H\colon \ring R\rightarrow H$
  denotes the unit of $H$ and $\tau_{H\otimes A}$ denotes the flip
  $h\otimes a\mapsto a\otimes h$.
\end{definition}
Notice that $\phi$ can be regarded as a map $H\rightarrow \End_{\ring
  R}(A) = C^1(A;A)$ satisfying certain conditions.
\begin{proposition}
  \label{actionimpliesdefsym}
  Given a Hopf algebra action $\phi$ of $H$ on $A$, the map $\Phi$
  defined by
  \begin{equation}
    h_0\otimes \ldots \otimes h_k\mapsto \mu_{\algebra A}^{(k)}\circ(\phi(h_0)\otimes
    \phi(h_1)\otimes \ldots\otimes \phi(h_k))
  \end{equation}
  is a deformation symmetry. Here $\mu_{\algebra A}^{(k)}$ denotes the
  $k$-th iteration $\mu_{\algebra A}\circ(\id\otimes \mu_{\algebra
    A})\circ\ldots\circ(\id^{\otimes k-1}\otimes \mu_{\algebra A})$ of
  $\mu_{\algebra A}$.
\end{proposition}
\begin{proof} 

  \leavevmode

  \noindent We prove this proposition by showing that $\Phi$ is a
  map of DGLA's.  Note that for $P=P_0\otimes \ldots\otimes
  P_k$, $Q=Q_0\otimes \ldots\otimes Q_l$ and $a=a_0\otimes
  \ldots\otimes a_{k+l}$ we have
  \begin{align*}
    \Phi(P\bullet Q)(a)
    &=
    \Phi\left(\sum_{i=0}^k(-1)^{il}P_0\otimes \ldots\otimes P_{i-1}\otimes \Delta^{(l)}(P_i)Q\otimes P_{i+1}\otimes \ldots\otimes P_k\right)(a)
    \\
    &=
    \sum_{i=0}^k(-1)^{il}
    \phi(P_0)(a_0)\cdot\ldots\cdot
    \phi(P^{(0)}_iQ_0)(a_i)\cdot\ldots\cdot\phi(P^{(l)}_iQ_l)(a_{i+l})\cdot
    \\
    &\qquad\cdot
    \phi(P_{i+1})(a_{i+1+l})\cdot\ldots\cdot\phi(P_k)(a_{k+l})
    \\
    &=
    \sum_{i=0}^k(-1)^{il}\phi(P_0)(a_0)
    \cdot\ldots\cdot
    \phi(P_i)(\phi(Q_0)(a_i)\cdot\ldots\cdot\phi(Q_l)(a_{l+i}))\cdot
    \\
    &\qquad\cdot \phi(P_{i+1}(a_{l+i+1})\cdot\ldots\cdot\phi(P_k)(a_{l+k})
    \\
    &=
    \sum_{i=0}^k(-1)^{il}\phi(P_0)(a_0)
    \cdot\ldots\cdot
    \phi(P_{i-1})(a_{i-1})\cdot\phi(P_i)(\Phi(Q)(a_{i}\otimes\ldots\otimes a_{i+l})\cdot
    \\
    &\qquad\cdot \phi(P_{i+1})(a_{l+i+1})\cdot \ldots\cdot\phi(P_{k})(a_{k+l})
    \\
    &=
    \sum_{i=0}^k(-1)^{il}\Phi(P)(a_0\otimes 
    \ldots\otimes a_{i-1}\otimes \Phi(Q)(a_i\otimes\ldots\otimes a_{i+l})\otimes a_{i+l+1}\otimes \ldots\otimes a_{k+l})
    \\
    &=
    \Phi(P)\{\Phi(Q)\}(a)
  \end{align*}
  where we have used the \emph{brace} notation, see e.g.
  \cite{GerstVoron1995}. Note that
  $[A,B]_G=A\{B\}-(-1)^{|A||B|}B\{A\}$. In short the above computation 
  boils down to
  \begin{equation}
    \Phi(P\bullet Q)
    =
    \Phi(P)\{\Phi(Q)\}.
  \end{equation} 
  Thus $\Phi$ respects the brackets and observing that $\Phi(1\otimes
  1)=\mu_{\algebra A}$ the proposition is proved.
\end{proof}
\begin{remark}
  Recall that $TH[1]$ is endowed with a brace algebra structure, as described in
  Remark~\ref{rem:braces}.  The above proposition actually follows
  from the fact that a Hopf algebra action $\phi$ of $H$ on
  $\mathscr{A}$ actually induces a brace algebra morphism
  \begin{equation}
    \Phi\colon TH[1]\longrightarrow C^\bullet(\mathscr{A}),
  \end{equation}
  where the braces on $C^\bullet(\mathscr{A})$ are defined as usual by 
  \begin{align*}
    P\{Q_1,\ldots, Q_r\}&(a_0\otimes \ldots\otimes a_{k+k_1+\ldots+k_r})
    =
    \sum_{0\leq i_1<i_2<\ldots<i_r\leq k}(-1)^{i_1k_1+i_2k_2+\ldots+i_rk_r}
    P(a_0\otimes \ldots\otimes 
    \\
    &\otimes
    Q_1(a_{i_1}\otimes 
    \ldots \otimes a_{i_1+k_1})\otimes a_{i_1+k_1+1}\otimes\ldots 
    \otimes a_{i_2+k_1-1}\otimes 
    Q_2(a_{i_2+k_1}\otimes \ldots\otimes a_{i_2+k_1+k_2})\otimes 
    \\
    &\ldots\otimes 
    Q_r(a_{i_r+k_1+\ldots+k_{r-1}}\otimes a_{i_r+k_1+\ldots+k_r})\otimes \ldots \otimes a_{k+k_1+\ldots+k_r}),
  \end{align*}
  for $P\in C^{k+1}(\mathscr{A};\mathscr{A})$, $Q_j\in
  C^{k_j+1}(\mathscr{A};\mathscr{A})$ for all $j$ and $a_i\in
  \mathscr{A}$ for all $i$. So we have
  \begin{equation} 
    \Phi(P\langle Q_1,\ldots, Q_r\rangle)
    =
    \Phi(P)\{\Phi(Q_1),\ldots, \Phi(Q_r)\}
  \end{equation}
  for all $P, Q_1,\ldots, Q_r\in TH[1]$. 
\end{remark}
%

% 
%Twisting the infinitesimal Poisson action
%
\subsection{Twisting Poisson actions}

A Lie bialgebra is a pair $(\lie g, \gamma)$ where $\lie g$ is a Lie
algebra and $\gamma$ is a $1$-cocycle, $ \gamma \colon \lie g \to \lie
g\wedge \lie g$.  In this paper we consider a particular class of Lie
bialgebras.  Recall that an element $r \in \lie g\wedge \lie g$ is
called $r$-matrix if it satisfies the Maurer--Cartan equation
$\Schouten{r, r}= 0$.  It can be proved that $r$-matrices always
induce a Lie bialgebra structure on $\lie g$, by setting $\gamma =
\Schouten{r, \argument }$. We refer to the pair $(\lie g, r)$ as
\emph{triangular Lie algebra}. For further details on Lie bialgebras
we refer to \cite{Kosmann-Schwarzbach2004}.
\vspace{0.3cm}

Let us consider a Lie algebra action $\varphi : \lie{g} \to
\Secinfty(TM)$.
\begin{definition}[Poisson action]
  The action $\varphi$ is Poisson if it satisfies
  \begin{equation}
    \D_\pi \varphi(X)
    =
    \varphi \wedge \varphi \circ \gamma (X)
  \end{equation}
  where $\D_\pi = \Schouten{\pi , \argument}$. 
\end{definition}
\begin{proposition}
  \label{prop:PoissonR}
  Let $\lie{g}$ be a finite dimensional triangular Lie algebra and
  $\varphi \colon \lie{g} \to \Secinfty(TM)$ a Lie algebra action.
  \begin{propositionlist}
  \item The bitensor $\pi$ defined as the image of $r$ via $\varphi$
    is a Poisson tensor.
  \item $\varphi$ is a Poisson action wrt $\pi = \varphi \wedge
    \varphi (r)$.
  \end{propositionlist}
\end{proposition}
\begin{proof}
  Let us consider the $r$-matrix $r = \frac{1}{2} \sum_{i,j} r^{ij}
  e_i \wedge e_j$ and define
  \begin{equation}
    \label{eq:PoissonR}
    \pi
    :=
    \frac{1}{2} \sum_{i,j} r^{ij} \varphi(e_i) \wedge \varphi(e_j).
  \end{equation}
  From $\Schouten{r, r} = 0$, using the fact that $\varphi$ is a Lie
  algebra morphism, it follows that $\Schouten{\pi, \pi} = 0$. The second
  claim is a straightforward computation.
\end{proof}
%
% We can regard Poisson actions in terms of Gerstenhaber
% algebras.  A morphism of differential Gerstenhaber
% algebras is a cochain map that respects the wedge product and the
% graded Lie bracket. 
It is easy to see that 
%$(\wedge^{\bullet}\lie g
%,\delta, \wedge, [\argument ,\argument])$ and $(\Secinfty
%(\wedge^{\bullet}TM), \D_\pi, \wedge, \Schouten{\argument , \argument}
%)$ are strong differential Gerstenhaber algebras and 
the notion of Poisson action can be extended to a morphism of
DGLA's
\begin{equation}
  \label{eq:PoissonActionExt}
  \wedge^{\bullet}\varphi
  : 
  (\wedge^{\bullet}\lie g ,\delta, [\argument ,\argument])
  \longrightarrow 
  (\Secinfty (\wedge^{\bullet} TM), \D_\pi, \Schouten{\argument , \argument}).
\end{equation}

We now show how one may use the formality of $\Dpoly{\mathfrak{g}}{}$
and $\Dpoly{}{}(M)$ to obtain a deformation symmetry of
$\mathscr{U}(\mathfrak{g})$ in $\Cinfty(M)$ given an infinitesimal
action $\varphi$ of $\mathfrak{g}$ on $M$. Thus, given an r-matrix of
$\mathfrak{g}$, we obtain deformations of all relevant structures and
we comment on these. In Section~\ref{sec:formality} we have obtained
the \emph{formal} DGLA's $\Dpoly{\mathfrak{g}}{}$ and $\Dpoly{}{}(M)$.
The classical $r$-matrix $r\in\Tpoly{\mathfrak{g}}{1}$ yields a
Maurer-Cartan element. Although it satisfies the Maurer-Cartan
equation $\Schouten{r,r}=0$ it is in fact not an MC element according
to our Definition~\ref{MCdef} as we have neglected to consider any
filtration. The filtration is needed since we go through formality,
which means we encounter infinite sums. More precisely, we obtain a
filtration by considering the formal power series ring $\R\hhbar$. 
Given a DGLA $\mathfrak{L}$ we denote the DGLA obtained by
extending scalars to the formal power series ring by
$\mathfrak{L}\hhbar$.  We consider these DGLA's as filtered by the
degree in $\hbar$. Note that, given $L_\infty$-morphisms of DGLA's
$\mathfrak{L}\rightarrow \mathfrak{L}'$ we obtain also
$L_\infty$-morphisms of the extended DGLA's
$\mathfrak{L}\hhbar\rightarrow\mathfrak{L}'\hhbar$.  In this way
$L_\infty$-quasi-isomorphism go to $L_\infty$-quasi-isomorphisms.

From Section~\ref{sec:formality} it follows that we
have a \emph{horse-shoe} diagram in the category of
$L_\infty$-algebras:
\begin{equation}
  \label{horse-shoe}
  \begin{tikzcd}[row sep=3em,column sep=4em]
    & \Tpoly{\mathfrak{g}}{}\hhbar
    \arrow[r, "\varphi"]\arrow[d, swap, "F_{\mathfrak{g}}"] 
    & \Tpoly{}{}(M)\hhbar \arrow[d,  "F_M"]
    \\
    & \Dpoly{\mathfrak{g}}{}\hhbar
    & \Dpoly{}{}(M)\hhbar.
  \end{tikzcd}
\end{equation}
Here the vertical maps $F_\mathfrak{g}$ and $F_M$ are
$L_\infty$-quasi-isomorphisms constructed as discussed in
Section~\ref{sec:formality} and the horizontal arrow is induced by
$\varphi$ as in \eqref{eq:PoissonActionExt}.
%\begin{equation}
%	X_1\wedge\ldots\wedge X_k\mapsto \varphi(X_1)\wedge\ldots\wedge\varphi(X_k).
%\end{equation}
%
Thus we obtain, given an r-matrix $r\in\mathfrak{g}\wedge\mathfrak{g}$, the MC elements:
\begin{equation}
  \begin{array}{l}
    \hbar r\in \Tpoly{\mathfrak{g}}{}\hhbar;
    \\
    \hbar\pi
    =
    (\hbar r)_\varphi\in \Tpoly{}{}(M)\hhbar;
    \\
    \rho_\hbar
    :=
    (\hbar r)_{F_\mathfrak{g}}\in \Dpoly{\mathfrak{g}}{}\hhbar
    \\
    B_\hbar
    :=
    (\hbar \pi)_{F_M}\in \Tpoly{\mathfrak{g}}{}\hhbar.
  \end{array}
\end{equation}

\begin{remark} 
  \label{rem:WrongCandidate}
  Notice that we obtain the deformed versions of $\Cinfty(M)$ and
  $\mathscr{U}(\mathfrak{g})$ by twisting by the Maurer-Cartan
  elements listed above. We would like then to obtain a deformation
  symmetry completing the square in the horse-shoe diagram above, such
  that it commutes.  Commutativity ensures that the two formal
  deformations of $\Cinfty(M)$ induced by $\hbar r$ by transporting it
  along the bottom or the top to $\Dpoly{}{}(M)\hhbar$ coincide. An
  obvious candidate for such a map would be
  \begin{equation}
    X_1\otimes\ldots\otimes X_k
    \mapsto 
    \varphi(X_1)\otimes\ldots\otimes\varphi(X_k),
  \end{equation}
  i.e. the map induced by the map of Lie-Rinehart pairs
  $(\mathfrak{g}, \R)\rightarrow (\Secinfty(TM),\Cinfty(M))$.  In
  other terms the deformation symmetry induced by the obvious Hopf
  algebra action of $\mathscr{U} (\mathfrak{g})$ on $\Cinfty(M)$
  through proposition \ref{actionimpliesdefsym}. However this map may
  not make the diagram commute. In fact it is the opinion of the
  authors that such commutation would involve some condition of
  compatibility of the connections used in defining these maps, see
  Remark~\ref{rem:choice}.
\end{remark}
We complete the horse-shoe diagram \eqref{horse-shoe} to a commuting
square by observing that $L_\infty$-quasi-isomorphisms are
\emph{invertible}. The following lemma is essentially contained in
\cite[Chapt.~10.4]{AO}. We remind the reader that the field underlying
$L_\infty$-algebras is of characteristic $0$.
\begin{lemma}
  \label{invert}
  Suppose $\mathfrak{L}$ and $\mathfrak{L}'$ are formal
  $L_\infty$-algebras and $f\colon H(\mathfrak{L})\rightarrow
  H(\mathfrak{L'})$ is an $L_\infty$-morphism, then there exists a
  lift $\tilde{f}\colon \mathfrak{L}\rightarrow \mathfrak{L}'$ of
  $f$. In other words there exists a commuting diagram
  \begin{equation}
    \begin{tikzcd}[row sep=3em,column sep=4em]
      & H(\mathfrak{L})
      \arrow[r, "f"]\arrow[d, swap, "i_l"] 
      & H(\mathfrak{L}') \arrow[d,  "i_r"]
      \\
      & \mathfrak{L} \arrow[r, swap, "\tilde{f}"] 
      & \mathfrak{L}'.
    \end{tikzcd}
  \end{equation}
  in the category of $L_\infty$-algebras such that the vertical arrows
  are quasi-isomorphisms.
\end{lemma}
\begin{proof}
  Note that we start by hypothesis with the horse-shoe
  \begin{equation}
    \begin{tikzcd}[row sep=3em,column sep=4em]
      & H(\mathfrak{L})
      \arrow[r, "f"]\arrow[d, swap, "i_l"] 
      & H(\mathfrak{L}') \arrow[d,  "i_r"]
      \\
      & \mathfrak{L}
      & \mathfrak{L'}.
    \end{tikzcd}
  \end{equation}
  in the category of $L_\infty$-algebras such that the vertical arrows
  are quasi-isomorphisms. As shown in \cite[Sect.~10.4.4]{AO} we can
  always find a quasi-inverse $l_i\colon\mathfrak{L}\rightarrow
  H(\mathfrak{L})$ of $i_l$ such that the induced maps in cohomology
  are inverse to each other. \ We define $\tilde{f}:=i_r\circ f\circ
  l_i$. Note that this already proves that $\tilde{f}\circ
  i_l=i_r\circ f$ in cohomology (and thus ``up to homotopy''). To get
  the stronger statement in the lemma we will need to consider the
  construction of the map $l_i$. This construction involves the notion
  of a so-called $\infty$-isomorphism. An $\infty$-isomorphism is a
  morphism of $L_\infty$-algebras $F$ such that $F_1$ is an
  \emph{isomorphism}. The main observations for the construction of
  $l_i$ are two-fold. First, the homotopy transfer theorem
  \cite[Sect.~10.3]{AO} yields an $L_\infty$-structure on
  $H(\mathfrak{L})$ which is unique up to $\infty$-isomorphism. It is
  obtained by picking a retraction of $\mathfrak{L}$ onto
  $H(\mathfrak{L})$, we may pick the retraction given by
  $i_l$. Secondly in section 10.4.2 of \cite{AO} it is shown that any
  $L_\infty$-algebra $\mathfrak{L}$ is $\infty$-isomorphic to the sum
  $H(\mathfrak{L})\oplus K$ where $K$ is an acyclic chain complex
  (with trivial $Q_0$, $Q_2$, $Q_3$ and so on). Now the construction
  of $l_i$ follows by the fact that $\infty$-isomorphisms are
  invertible (shown in section 10.4.1 of \cite{AO}). Our lemma follows
  from the fact that, for a formal $L_\infty$-algebra the
  $L_\infty$-structure induced on $H(\mathfrak{L})$ by homotopy
  transfer equals the canonically induced structure up to
  $\infty$-isomorphism. Thus we may simply consider the splittings
  $\mathfrak{L}\simeq H(\mathfrak{L})\oplus K_{\mathfrak{L}}$ and
  $\mathfrak{L}'\simeq H(\mathfrak{L}')\oplus K_{\mathfrak{L}'}$,
  where $\simeq$ means $\infty$-isomorphism, given by the retractions
  induced by $i_l$ and $i_r$. Then the map $\tilde{f}$ simply maps
  $H(\mathfrak{L})$ to $H(\mathfrak{L}')$ by $f$.
\end{proof}
\begin{corollary}
\label{twistaction}  The diagrams 
  \begin{equation}
    \label{square}
    \begin{tikzcd}[row sep=3em,column sep=4em]
      & \Tpoly{\mathfrak{g}}{}\hhbar
      \arrow[r, "\varphi"]\arrow[d, swap, "F_{\mathfrak{g}}"] 
      & \Tpoly{}{}(M)\hhbar \arrow[d,  "F_M"]
      \\
      & \Dpoly{\mathfrak{g}}{}\hhbar \arrow[r, "\tilde{\varphi}"]
      & \Dpoly{}{}(M)\hhbar
    \end{tikzcd}
  \end{equation}
  and
  \begin{equation}
    \label{twistsquare}
    \begin{tikzcd}[row sep=3em,column sep=4em]
      & \Tpoly{\mathfrak{g}}{\hbar r}\hhbar
      \arrow[r, "\varphi^{\hbar r}"]\arrow[d, swap, "F^{\hbar r}_{\mathfrak{g}}"] 
      & \Tpoly{}{\hbar\pi}(M)\hhbar \arrow[d,  "F^{\hbar\pi}_M"]
      \\
      & \Dpoly{\mathfrak{g}}{\rho_\hbar}\hhbar \arrow[r, "\tilde{\varphi}^{\rho_\hbar}"]
      & \Dpoly{}{B_\hbar}(M)\hhbar
    \end{tikzcd}
  \end{equation}
  commute.
\end{corollary}
\begin{proof}
  Applying the above lemma to our situation we immediately find that
  the diagram \eqref{square} commutes. Also, by applying the results
  of Section~\ref{sec:Twisting} we obtain diagram \eqref{twistsquare},
  which commutes thanks to Remark~\ref{comptwist}.
\end{proof}

\subsection{Twisted structures}

In the following we show that the twisted complexes obtained above are
coming from a formal deformation quantization of $\Cinfty(M)$ (in the
case of $\Dpoly{}{B_\hbar}(M)$) and a deformation of
$\mathscr{U}(\mathfrak{g})$ into a quantum group (in the case of
$\Dpoly{\mathfrak{g}}{\rho_\hbar}$).
\begin{proposition}
  \label{defHoch}
  There is a formal deformation quantization $\mathscr{A}_\hbar$ of
  $(\Cinfty(M),\pi)$ such that
  \begin{equation}
    \Dpoly{}{B_\hbar}(M)
    \hookrightarrow 
    C^\bullet(\algebra{A}_\hbar,\algebra{A}_\hbar),
  \end{equation}
  i.e. $\Dpoly{}{B_\hbar}(M)$ is a subcomplex of the Hochschild
  cochain complex of $\algebra{A}_\hbar$.
\end{proposition}
\begin{proof}
  Note that the Maurer--Cartan equation \eqref{eq:MC} matches exactly
  the associativity condition of $m + B_\hbar$, where $m$ denotes the
  pointwise multiplication in $\Cinfty(M)\hhbar$. Since
  \begin{equation}
    B_\hbar
    =
    \sum_{n\geq 1}\frac{1}{n!}(F_M)_n(\hbar\pi,\ldots,\hbar\pi)
    =
    \sum_{n\geq 1}\frac{\hbar^n}{n!}(F_M)_n(\pi,\ldots,\pi),
  \end{equation}
  we see that $m+B_\hbar$ defines a deformation quantization
  $\algebra{A}_\hbar$. The differential on the Hochschild complex
  $C^\bullet(\algebra A, \algebra A)$ is given by taking the
  Gerstenhaber bracket with the multiplication for any associative
  algebra $\algebra A$.  Thus the twisted differential on
  $\Dpoly{}{B_\hbar}(M)$ coincides with the differential of
  $C^\bullet(\algebra{A}_\hbar,\algebra{A}_\hbar)$.  Finally we recall
  that
  \begin{equation}
    B_\hbar
    =
    \sum_{n\geq 1}\frac{\hbar^n}{n!}(F_M)_n(\pi,\pi,\ldots,\pi)
    =
    \hbar(F_M)_1(\pi)  
    \hspace{0.4cm}\mbox{mod}\hspace{0.2cm}\hbar^2.
  \end{equation}
  Since $(F_M)_1$ is a quasi-isomorphism we find that the alternating
  part of $B_\hbar$ is $\hbar\pi$ modulo $\hbar^2$ and
  \begin{equation}
    \frac{[f,g]_{\star}}{\hbar}
    =
    \pi(f,g)+O(\hbar)
  \end{equation}
  for all $f,g\in \Cinfty(M)$. Here $[\argument, \argument]_\star$
  denotes the commutator bracket of $\algebra{A}_\hbar$.
\end{proof}
Thus we obtain in particular the deformation quantization
$\algebra{A}_\hbar$.  On the other hand we also obtain the MC element
$\rho_\hbar$ in $\Dpoly{\mathfrak{g}}{}$. This yields the
coproduct $\Delta_{J_\hbar}$ (where $J_\hbar=1\otimes 1 +\rho_\hbar)$
thus establishing a quantum group, obtained by quantization of the Lie
bialgebra $(\mathfrak{g},\Schouten{r,\argument})$.
Finally we obtain the following theorem as a corollary of
Prop.~\ref{prop:twistJ} and Corollary~\ref{twistaction}. 

\begin{theorem}
  Suppose $(\lie g, r)$ is a triangular Lie algebra and $\varphi\colon
  \lie g\rightarrow \Secinfty(TM)$ is an action on the manifold
  $M$. Then there exist a formal deformation quantization
  $\mathscr{A}_\hbar$ of $(M,\pi)$ (where
  $\varphi\wedge\varphi(r)=\pi$) and a quantization
  $\mathscr{U}_\hbar(\lie g)$ of $(\lie g, r)$ which allow a
  deformation symmetry
  \begin{equation} 
    \tilde{\varphi}^{\rho_\hbar}\circ\mathscr{J}
    \colon 
    \mathscr{U}_\hbar(\lie g)_{poly}\longrightarrow C(\mathscr{A}_\hbar;\mathscr{A}_\hbar)
  \end{equation}
\end{theorem}
%

%\begin{tikzcd}%[row sep=3em,column sep=4em]
%	& (\wedge^\bullet \lie g, 0) 
 %   \arrow[r, "\varphi"]\arrow[d, swap, " "] 
 %   & (\Tpoly{}  (M), 0) \arrow[d,  " "]
 %   \\
 %   & (T \algebra{U}(\lie g), \Delta ) \arrow[r, "\hat{\varphi}"]
 %   & (\Dpoly{}(M), \partial)
%\end{tikzcd}

\subsection{Comparison with Drinfeld's construction}

%What we proved above motivates the following definition.
%\begin{definition}[Quantum action]
 % Let $\algebra A_\hbar$ be a formal deformation quantization of
 % $(\Cinfty(M), \pi)$ and $\algebra U(\lie g)$ a quantum group
 % corresponding to the Lie bialgebra $\lie g$. A quantum action is a
 % morphism $\varphi_\hbar \colon \algebra U(\lie g) \to
  %C^1(\algebra{A}_\hbar,\algebra{A}_\hbar)$ such that the
  %corresponding map $T \algebra U(\lie g) \to
  %C^\bullet(\algebra{A}_\hbar,\algebra{A}_\hbar)$ is a DGLA morphism.
%\end{definition}

First, let us briefly recall the original construction of Drinfeld
(see
\cite{aschieri.schenkel:2014a,drinfeld:1989a,giaquinto.zhang:1998a}). Consider
a formal twist $J$ on $\mathscr{U}(\mathfrak{g})\hhbar$ and a generic
$\mathscr{U}(\lie{g})$-module algebra $\algebra A$. Drinfeld proved
that we can then always define an associative star product on
$\algebra A$. In particular, consider $\algebra A = \Cinfty(M)$ with
pointwise multiplication $m$.  Given a Lie algebra action $\varphi:
\lie g \to \Secinfty(TM)$ we obtain a Hopf algebra action
\begin{equation}
  \label{eq:Action}
  \acts
  \colon
  \mathscr{U}(\lie{g}) \times \Cinfty(M)
  \longrightarrow
  \Cinfty(M),
\end{equation}
which makes $\Cinfty(M)$ into a left $\mathscr{U}(\lie{g})$-module
algebra. More precisely, $X \acts f = \Lie_{\varphi(X)} f$ where
$\Lie$ denotes the Lie derivative. The action $\acts$ extends to
formal power series 
\begin{equation}
  \acts \colon
  \mathscr{U}(\lie{g})\hhbar \times \Cinfty(M)\hhbar\longrightarrow
  \Cinfty(M)\hhbar.
\end{equation}
Thus the product defined by
\begin{align}
  \label{eq:StarTwist}
  f\star_J g
  =
  m(J\acts(f\tensor g))
  =
  m (\Lie_{\varphi \tensor \varphi(J)}(f \tensor g))
\end{align}
for $f, g \in \Cinfty(M)\hhbar$ is a star product.  The classical
limit of \eqref{eq:StarTwist} is given by
\begin{equation}
  \label{eq:ThePoissonStructure}
  \{f, g\}
  =
  m (r \acts (f \tensor g)),
\end{equation}
where $r:=\frac{J-\tau(J)}{\hbar}|_{\hbar=0}$ is the $r$-matrix
associated to the twist $J$, here $\tau$ denotes the flip $X\otimes
Y\mapsto Y\otimes X$.  It is important to underline that the deformed
algebra $(\Cinfty (M)\hhbar, \star_J)$ is then a module-algebra for
the quantum group:
\begin{equation}
  \UE_J(\lie{g}) 
  := (\UE(\lie{g})\hhbar, \Delta_J).
\end{equation}
In other words, $\acts$ is a Hopf agebra action of the twisted Hopf
algebra $\UE_J(\lie{g})$ on $(\Cinfty (M)\hhbar, \star_J)$.

%Let us note that the notation $H_{poly}$ for a Hopf algebra $H$
%  actually derives from the notation $\Dpoly{\mathfrak{g}}{}$ and in
%  fact, if we consider the Hopf algebra $\mathscr{U}(\mathfrak{g})=H$,
%  then $H_{poly}=\Dpoly{\mathfrak{g}}{}$. 

Thus, the construction takes a formal twist $J\in
\mathscr{U}(\mathfrak{g})\hhbar$ and an infinitesimal action of $\lie
g$ on $M$ as input and produces a deformation quantization $\algebra
A_\hbar$ together with an action of the quantum group
$\mathscr{U}_J(\lie g)$ on it. 
%This leaves in the middle how one obtained a formal twist to begin with. 
In our approach one starts with
an $r$-matrix $r\in\lie g\wedge\lie g$ and an infinitesimal action of
$\lie g$ on $M$ and obtains a formal twist $J$ and a deformation
symmetry $\tilde{\varphi}$. These then also yield a deformation
quantization given by
\begin{equation}
  f\star_r g=\tilde{\varphi}(J)(f\otimes g),
\end{equation}
and another deformation symmetry $\tilde{\phi}^{\rho_\hbar}\circ\mathscr{J}$ 
of the quantum group $\UE_J(\lie g)$ in the deformed algebra $\algebra A_\hbar$.
 %
  %We will
  %show in the following that the formality quasi-isomorphism
  %constructed in the previous sections lead to deformation symmetries
  %that are not coming from Hopf algebra actions.
The main difference of the two approaches is that the formal twist is taken as
given in Drinfeld's approach while we obtain it through quantization of an $r$-matrix. The trade-off
is however that we do not obviously obtain an action of a quantum
group anymore, instead we obtain the deformation symmetry. A direct comparison of the two approaches will be
nontrivial and it implies a study of the compatibility condition
between connections mentioned in Remark \ref{rem:WrongCandidate}. In
particular it is of interest whether the process of quantization so
obtained can be made functorial for equivariant maps between the
manifolds. We will come back to this in a future project.

{
  \footnotesize
  \renewcommand{\arraystretch}{0.5}
%  \bibliographystyle{chairx}
%  \bibliography{dqarticle,dqbook,dqprocentry,dqproceeding,dqthesis,misc,script,preprints,notes,twist}

\begin{thebibliography}{10}
\bibitem {Aguiar}
\textsc{Aguiar, M.: }\newblock \emph{Infinitesimal Bialgebras, Pre-Lie and Dendriform Algebras}.
\newblock Hopf Algebras, Dekker New York (2004), 1 - 33.

\bibitem {Arnal2007}
\textsc{Arnal, D., Dahmene, N., Tounsi, K.: }\newblock \emph{Poisson Action and Formality}.
\newblock Lett. Math. Phys.  \textbf{82}.2 (2007), 177-189.


\bibitem {aschieri.schenkel:2014a}
\textsc{Aschieri, P., Schenkel, A.: }\newblock \emph{Noncommutative connections
  on bimodules and {D}rinfel'd twist deformation}.
\newblock Adv. Theor. Math. Phys.  \textbf{18}.3 (2014), 513--612.

\bibitem {bayen.et.al:1978a}
\textsc{Bayen, F., Flato, M., Fr{{\o}}nsdal, C., Lichnerowicz, A., Sternheimer,
  D.: }\newblock \emph{Deformation Theory and Quantization}.
\newblock Ann. Phys.  \textbf{111} (1978), 61--151.


\bibitem {bieliavsky.gayral:2015a}
\textsc{Bieliavsky, P., Gayral, V.: }\newblock \emph{Deformation Quantization
  for Actions of {K}{\"a}hlerian Lie Groups}, vol. 236.1115 in \emph{Memoirs of
  the American Mathematical Society}.
\newblock American Mathematical Society, Providence, RI, 2015.

\bibitem {Calaque2005}
\textsc{Calaque, D.: }\newblock \emph{{Formality for {L}ie algebroids}}.
\newblock Comm. Math. Phys.  \textbf{257}.3 (2005)

\bibitem {Calaque2006}
\textsc{Calaque, D.: }\newblock \emph{{Quantization of formal classical dynamical r-matrices: the reductive case}}.
\newblock Adv. Math.  \textbf{204},1 2006.

\bibitem {chari.pressley:1994a}
\textsc{Chari, V., Pressley, A.: }\newblock \emph{A Guide to Quantum Groups}.
\newblock Cambridge University Press, Cambridge, 1994.

\bibitem {Niek}
\textsc{de~Kleijn, N.: }\newblock \emph{Actions on Deformation Quantization and
  an Equivariant Algebraic Index Theorem}.
\newblock PhD thesis, University of Copenhagen, 2016.

\bibitem {drinfeld:1983a}
\textsc{Drinfeld, V.~G.: }\newblock \emph{On constant quasiclassical solutions
  of the Yang-Baxter quantum equation}.
\newblock Sov. Math. Dokl.  \textbf{28} (1983), 667--671.

\bibitem {drinfeld:1983b}
\textsc{Drinfeld, V.~G.: }\newblock \emph{Hamiltonian structures on Lie groups, 
Lie bialgebras, and the geometric meaning of the classical Yang-Baxter equations}.
\newblock Sov. Math. Dokl.  \textbf{27} (1983), 285--287.

\bibitem {drinfeld:1988a}
\textsc{Drinfeld, V.~G.: }\newblock \emph{Quantum groups}.
\newblock J. Sov. Math.  \textbf{41} (1988), 898--918.

\bibitem {drinfeld:1989a}
\textsc{Drinfeld, V.~G.: }\newblock \emph{Quasi-Hopf algebras}.
\newblock Algebra i Analiz  \textbf{1} (1989), 114--148.

\bibitem {Dolgushev2005a}
\textsc{Dolgushev, V.~A.: }\newblock \emph{Covariant and equivariant formality
  theorems}.
\newblock Adv. Math.  \textbf{191}.1 (2005), 147--177.

\bibitem {Dolgushev2005}
\textsc{Dolgushev, V.~A.: }\newblock \emph{A Proof of Tsygan's Formality
  Conjecture for an Arbitrary Smooth Manifold}.
\newblock PhD thesis, MIT, 2005.

\bibitem {esposito:2015a}
\textsc{Esposito, C.: }\newblock \emph{Formality theory. From {P}oisson
  structures to deformation quantization.}
\newblock Springer-Verlag, Heidelberg, Berlin, 2015.



\bibitem {enriquez.etingof:2005}
\textsc{Enriquez, B., Etingof, P. : }\newblock \emph{Quantization of classical dynamical r-matrices with nonabelian base}.
\newblock Comm. Math. Phys.  \textbf{254}, 3 (2005), 603–650.

\bibitem {sisters:2017b}
\textsc{Esposito, C., de Kleijn, N.: }\newblock \emph{Dolgushev proof in global setting.}
\newblock In preparation.

\bibitem {ESW2016}
\textsc{Esposito, C., Schnitzer, J., Waldmann, S.: }\newblock \emph{An Universal Construction of Universal Deformation Formulas, Drinfeld Twists and their Positivity .}
\newblock https://arxiv.org/1608.00412.

\bibitem {etingof.schiffmann:1998a}
\textsc{Etingof, P., Schiffmann, O.: }\newblock \emph{Lectures on Quantum
  Groups}.
\newblock International Press, Boston, 1998.

\bibitem {fedosov:1994a}
\textsc{Fedosov, B.~V.: }\newblock \emph{A Simple Geometrical Construction of
  Deformation Quantization}.
\newblock J. Diff. Geom.  \textbf{40} (1994), 213--238.

\bibitem {Fedosov1998}
\textsc{Fedosov, B.: }\newblock \emph{{Non Abelian Reduction in deformation
  quantization}}.
\newblock Lett. Math. Phys.  \textbf{43}.2 (1998), 137--154.

\bibitem {gerstenhaber:1964a}
\textsc{Gerstenhaber, M.: }\newblock \emph{On the Deformation of Rings and
  Algebras}.
\newblock Ann. Math.  \textbf{79} (1964), 59--103.

\bibitem {GerstVoron1995}
\textsc{Gerstenhaber, M., Voronov, A.:}\newblock \emph{Higher Operations 
on the Hochschild complex}
\newblock Funct. Anal. Appl. \textbf{29} (1995), 1--6.

\bibitem {giaquinto.zhang:1998a}
\textsc{Giaquinto, A., Zhang, J.~J.: }\newblock \emph{Bialgebra actions,
  twists, and universal deformation formulas}.
\newblock J. Pure Appl. Algebra  \textbf{128}.2 (1998), 133--152.

\bibitem {GradyGwilliam}
\textsc{Grady, R., Gwilliam, O.: }\newblock \emph{Lie Algebroids and
  $L_\infty$-spaces}.
\newblock arXiv:1604.00711, 2016.

\bibitem {kontsevich:2003a}
\textsc{Kontsevich, M.: }\newblock \emph{Deformation Quantization of {P}oisson
  manifolds}.
\newblock Lett. Math. Phys.  \textbf{66} (2003), 157--216.

\bibitem {Kosmann-Schwarzbach2004}
\textsc{Kosmann-Schwarzbach, Y.: }\newblock \emph{{Lie bialgebras, {P}oisson
  {L}ie groups and dressing transformations}}.
\newblock In: \textsc{Kosmann-Schwarzbach, Y., Grammaticos, B., Tamizhmani,
  K.~M. (eds.): }\newblock \emph{{Integrability of Nonlinear Systems}}.
  Springer, 2004.

\bibitem {AO}
\textsc{Loday, J.-L., Valette, B.: }\newblock \emph{Algebraic Operads}.
\newblock \emph{GMW 346}.
\newblock Springer-Verlag, 2012.

\bibitem {Markl100}
\textsc{Markl, M.: }\newblock \emph{Deformation Theory of Algebras and Their
  Diagrams}.
\newblock In: \emph{CBMS Regional Conf. Ser. in Math. Number 116},   page 100.
  Conf. Board Math. Sci., Washington, DC, 2011.

\bibitem {Moerdijk2010}
\textsc{Moerdijk, I., Mr\v{c}un, J.: }\newblock \emph{On the Universal
  Enveloping Algebra of a Lie Algebroid}.
\newblock Proc. Amer. Math. Soc.  \textbf{138}.9 (2010), 3135–3145.

\bibitem {MoerdijkReyes}
\textsc{Moerdijk, I., Reyes, G.: }\newblock \emph{Models for Smooth
  Infinitesimal Analysis}.
\newblock Springer-Verlag, 1991.

\bibitem {Quillen69}
\textsc{Quillen, D.: }\newblock \emph{Rational homotopy theory}.
\newblock Ann. of Math.  \textbf{90}.2 (1969), 205–295.

\bibitem {Semenov1985}
\textsc{Semenov-Tian-Shansky, M.A.: }\newblock \emph{Dressing Transformations and Poisson Group Actions}.
\newblock Publ. RIMS Kyoto Univ.  \textbf{21} (1985), 1237--1260.

\bibitem {Sharygin2016}
\textsc{Sharygin, G.:}\newblock \emph{ Deformation quantization and the action of Poisson vector fields }.
\newblock  arXiv:1612.02673v1

\bibitem {waldmann:2007a}
\textsc{Waldmann, S.: }\newblock \emph{Poisson-{G}eometrie und
  {D}eformationsquantisierung. {E}ine {E}inf{\"u}hrung}.
\newblock Springer-Verlag, Heidelberg, Berlin, New York, 2007.

\bibitem {xu:2002a}
\textsc{Xu, P.: }\newblock \emph{Triangular dynamical r-matrices and quantization}.
\newblock Adv. Math . \textbf{166}.1 (2002), 1–49.
\end{thebibliography}
%}

}

\end{document}